\documentclass[a4paper,english,oneside,11pt]{article}
\usepackage[T1]{fontenc}
\usepackage[utf8]{inputenc}
\usepackage{graphicx}
\usepackage{amsmath, amsfonts, amsthm}
\usepackage{microtype}
\usepackage{geometry}
\numberwithin{equation}{section}

\usepackage{tikz}
\usepackage[backend=biber,doi=true,style=numeric,maxnames=3,maxbibnames=9]{biblatex}
\theoremstyle{plain}
\newtheorem*{theorem*}{Theorem}
\newtheorem*{proposition*}{Proposition}
\newtheorem{theorem}{Theorem}[section]

\newtheorem{thm}[theorem]{Theorem}

\theoremstyle{remark}
\newtheorem*{remark}{Remark}

\theoremstyle{definition}
\newtheorem{definition}[theorem]{Definition}
\newtheorem{defn}[theorem]{Definition}

\newcommand{\D}{\mathcal{D}}
\newcommand{\R}{\mathbb{R}}

\newcommand{\E}{\mathbb{E}}

\renewcommand{\L}{\mathcal{L}}

\DeclareMathOperator{\dcov}{dcov}
\DeclareMathOperator{\dCov}{dCov}
\DeclareMathOperator{\dcor}{dcor}
\DeclareMathOperator{\dCor}{dCor}
\DeclareMathOperator{\dVar}{dVar}
\DeclareMathOperator{\dvar}{dvar}
\DeclareMathOperator{\PL}{PL}
\DeclareMathOperator{\SW}{SW}

\newcommand{\ie}{}
\def\ie/{i.e.}
\newcommand{\eg}{}
\def\eg/{e.g.}

\title{Same But Different\\
  \large Distance Correlations Between Topological Summaries}
\author{Katharine Turner and Gard Spreemann}
\date{\today}

\begin{document}
\maketitle

\begin{abstract}

Persistent homology allows us to create topological summaries of
complex data. In order to analyse these statistically, we need to
choose a topological summary and a relevant metric space in which this
topological summary exists. While different summaries may contain the
same information (as they come from the same persistence module), they
can lead to different statistical conclusions since they lie in
different metric spaces. The best choice of metric will often be
application-specific. In this paper we discuss distance correlation,
which is a non-parametric tool for comparing data sets that can lie in
completely different metric spaces. In particular we calculate the
distance correlation between different choices of topological
summaries. We compare some different topological summaries for a
variety of random models of underlying data via the distance
correlation between the samples. We also give examples of performing
distance correlation between topological summaries and other scalar
measures of interest, such as a paired random variable or a parameter
of the random model used to generate the underlying data. This article
is meant to be expository in style, and will include the definitions
of standard statistical quantities in order to be accessible to
non-statisticians.

\end{abstract}

\section{Introduction}

The development and application of statistical theory and methods
within Topological Data Analysis (TDA) are still in their
infancies. The main reason is that distributions of topological
summaries are harder to study than distributions of real numbers, or
of vectors. Complications arise both from the geometry of the spaces
that the summaries lie in, and, more importantly, from the complete
lack of nice parameterised families which one could expect the
distributions of topological summaries to follow. Even when the
distribution of filtrations of topological spaces is parametric,
topological summaries do not necessarily preserve distributions in any
meaningful way, so the resulting topological summaries will generally
not be in the form of a tractable exponential family. Effectively none
of the methods from basic statistics can be directly applied, at least
not without significant caveats and great care. We should instead turn
to the world of non-parametric statistics, in which the methods are
usually distribution-free and can sometimes be applied to random
elements lying in quite general metric spaces.

A quintessential example of the challenges faced when improving the
statistical rigour in TDA is that of correlation. The Pearson
correlation coefficient is the correlation, the ``r-value'', taught in
every introductory course on statistics. However, it is only defined
for real-valued functions, and inference involving Pearson correlation
often assumes the variables follow normal distributions. It is very
much a method from parametric statistics. It measures the strength of
the linear relationship between normally distributed random
variables. Here the parametric families are the normal distributions
with the mean and covariance as the parameters, and the correlation
coefficient is a straightforward function of the covariance
matrix. The Pearson correlation is very useful and appropriate if the
distributions are normal.  However, that is a very big ``if''; and one
that will rarely hold for topological summaries.

Thankfully for practitioners of TDA, correlation as a concept is not
defined by the formula of the Pearson correlation coefficient, but
rather should be thought of more philosophically as some quantity
measuring the extent of interdependence of random variables. This
research is the result of a treasure hunt within the field of
non-parametric statistics for an appropriate notion of correlation
applicable to topological summaries. Our finding was distance
correlation.  Effectively it considers the correlations between
pairwise distances (appropriately recentred) instead of the raw
values. This makes it applicable to distributions over any pair of
metric spaces.

Distance correlation is a distribution-free method and exemplifies a
non-parametric approach. It can detect relationships between variables
that are not linear, and not even monotonic. If the variables are
independent, then the distance correlation is zero. In the other
direction, if the metric spaces are of strong negative type, then a
distance correlation of zero implies the variables are
independent. This is true for any joint distribution. In contrast, we
can only conclude from a Pearson correlation coefficient of zero that
the variables are independent if we assume that the joint distribution
is bivariate normal.

There are two take-home messages. The first is that distance
correlation is a useful tool in the statistical analysis of
topological summaries. The current exposition serves as an
introduction to the potential of distance correlation for statistical
analysis in TDA. In section~\ref{sec:future} we outline some further
opportunities that distance correlation can offer. The second message
is the simple observation that the choice of topological summary
statistic matters. A responsible topological data analyst should
carefully consider which is the most appropriate topological
summary. A better choice is one where the pairwise distances better
reflect the differences of interest in the raw data. This will be
domain- and application-specific. There may be other considerations
for the choice of topological summary in terms of the statistical
methods available, computational complexity and inference power, but
this is beyond the scope of this discussion.

\section{Background theory}

We summarize the relevant basic concepts from TDA, the statistical
analysis of topological summaries, and of metric spaces of strong
negative type.

\subsection{Topological summary statistics}

TDA is usually concerned with analysing complex and hard-to-visualize
data. This data may have complicated geometric or topological
structure, and one creates a family of topological spaces which can
then be studied using algebraic-topological methods in order to reveal
information about said structure. We call a family of spaces
$(K_a)_{a\in\R}$ such that $K_a\subset K_b$ whenever $a\leq b$ a
\emph{filtration}. The inclusion $K_a \subset K_b$ for $a \leq b$
induces a homomorphism $H_k(K_a)\to H_k(K_b)$ between the homology
groups. The \emph{persistent homology group} is the image of
$H_k(K_a)$ in $H_k(K_b)$. It encodes the $k$-cycles in $K_a$ that are
independent with respect to boundaries in $K_b$, \ie/
\begin{equation}
  H_k(a,b) := Z_k(K_a)/ (B_k(K_b) \cap Z_k(K_a)).
\end{equation}
where $Z_k$ and $B_k$ are, respectively, the kernel and the image of
the $k$'th boundary map in the given homology theory.

%Algorithms for computing persistent homology from a given filtration are quite simple in their most basic form \cite{edelsbrunner_topological_2002,zomorodian_computational_2009}, and are now implemented efficiently in a number of freely-downloadable packages~\cite{morozov_dionysus_2013,tausz_javaplex_2014,bauer_phat_2014,nanda_perseus_2015}. 

Under very general assumptions on the filtration, and assuming one
works with coefficients in a field, persistent homology is fully
described by two equivalent representations: the barcode and the
persistence diagram. The barcode is a collection of intervals $[b,d)$
each representing the first appearance (``birth''), $b$, and first
disappearance (``death''), $d$, of a persistent homology class. This
collection of intervals satisfies the condition that for every $b\leq d$,
the number of intervals containing $[b,d)$ is
$\dim(H_k(b,d))$. The corresponding persistence diagram is the
multi-set of points in the plane with birth filtrations as one coordinate
and death filtrations as the other.\footnote{We will consider only persistent homology with intervals with finite death.}

A \emph{summary statistic} is a object that is used to summarise a set
of observations, in order to communicate the largest amount of
information as simply as possible. Simple examples in the case of
real-valued distributions include the mean, the variance, the median
and the box plot. Many summary statistics in TDA are created via
persistent homology. We have a filtration of topological spaces built
from our observations, and by applying persistent homology we can
summarise this filtration in terms of the evolution of
homology. Notably, one creates a summary from a single complex object,
whether it be a point cloud, a graph, etc.

There are now a wide array of topological summaries that can be
computed directly from a persistence diagram or barcode. Each of these
is a different expression of the persistent homology in the form of a
topological summary statistic. The practitioner wanting to perform
statistical analysis using topological summaries needs to choose which
type of summary to represent their data with, as well as the metric on
the space where that summary takes values. For some of these different
topological summaries there are parameters to choose which play roles
like bandwidth, and some depend on a choices like norm order (for us,
$p\in \{1,2,\infty\}$) akin to choosing $p$ in the $L^p$ distance for
function spaces. In addition, there are topological summaries not
based on persistent homology, such as simplex count functions.

In this paper we will consider a range of different topological
summaries and distances defined on them, namely:
\begin{itemize}
	\item Persistence diagrams, with Wasserstein distances for $p=1,2, \infty$
	\item Persistence landscapes~\cite{Landscapes}, with $L^p$ distances for $p=1,2, \infty$
	\item Persistence scale space kernel~\cite{Heat} for
          two different bandwidths, with $L^2$ distances
	%\item Accumulation Persistence Function, with $L^2$ distances
	\item Betti and Euler characteristic curves, with $L^p$ distances for $p=1,2$
	\item Sliced Wasserstein kernel~\cite{Sliced} distance
\end{itemize}

Note that all of these topological summaries can be computed from the
persistence diagram and that with the exception of the Betti and Euler
curves, which collapse information, they all are distance functions of
the information provided in the original persistence diagram. In this
sense they are the ``same''. It is merely the metric space structure
that is different.

It is worth noting that the above list is by no means an exhaustive
list of topological summaries. Other examples include the persistent
homology rank function~\cite{rank}, the accumulation persistence
function~\cite{apf}, the persistence weighted Gaussian kernel
\cite{PWGK}, the persistence Fisher kernel~\cite{fisher}, using
tangent vectors from the mean of the square root framework with
principal geodesic analysis~\cite{Anirudh}, using points in the
persistence diagrams as roots of a complex polynomial for
concatenated-coefficient vector representations~\cite{comparing}, or
using distance matrices of points in persistence diagrams for
sorted-entry vector representations~\cite{Carriere}. Notably, most of
these are functional summaries with an $L^2$ metric or lie in a
reproducing kernel Hilbert space. Analogous arguments to those for the
persistence scale space discussed later could be used to show that
many of them lie in metric spaces of strong negative type as a
corollary of being separable Hilbert spaces.

\subsection{Distance correlation}

A \emph{random element} is a map from a probability space $\Omega$ to
a set $\mathcal{X}$. Its distribution is the pushforward measure on
$\mathcal{X}$. Given two random elements $X:\Omega \to \mathcal{X}$
and $Y:\Omega \to \mathcal{Y}$, one can consider the paired samples
$(X,Y): \Omega \to \mathcal{X}\times \mathcal{Y}$. This has a
\emph{joint distribution} on $\mathcal{X}\times \mathcal{Y}$. The
\emph{marginal distributions} for this joint distribution are the
pushforwards via the projection maps onto each of the coordinates. An
important notion in statistics is whether two variables are
\emph{independent}. This occurs precisely when the joint distribution
is the product of the marginal distributions.

The most common measure of correlation between two random variables is
the Pearson correlation coefficient. It is defined as the covariance
of the two variables divided by the product of their standard
deviations. For paired random variables $X, Y$, the Pearson
correlation coefficient is defined by
\begin{equation*}
  \rho_{X,Y}=\frac{\E[(X-\overline{X})(Y-\overline{Y})]}{\sigma_X \sigma_Y},
\end{equation*}
where $\overline{X}, \overline{Y}$ are the means of $X$ and $Y$,
$\sigma_X,\sigma_Y$ their standard deviations, and $\E$ denotes
expectation.  Note that if $X$ and $Y$ are independent, then
$\E[(X-\overline{X})(Y-\overline{Y})]=\E[(X-\overline{X})]\E[(Y-\overline{Y})]=0$. This
implies that a non-zero correlation is evidence of a lack of
independence, and hence the variables are related somehow (though
possible only indirectly).

The Pearson correlation is designed to analyse bivariate Gaussian
distributions. In this case, a correlation of $0$ implies that
the variables are independent. Furthermore, Pearson correlation
determines the ellipticity of the distribution. We can calculate the
Pearson correlation for more general distributions, but in that case
it detects linear relationships, and nonlinear relationships can be
lost. Some examples of the Pearson correlation coefficient are
illustrated in Figure~\ref{fig:PearsonvsDistance}. Any test using the
correlation coefficient (such as significance testing) depends on the
bivariate Gaussian assumption.

In parametric statistics we make some assumption about the parameters
(defining properties) of the population distribution(s) from which the
data are drawn, while in non-parametric statistics we do not make such
assumptions. Given the lack of parametric families of topological
summary statistics, it makes sense to consider non-parametric
methods. One option when studying real-valued random variables which
are not normally distributed, or when the relationship between the
variables is not linear, is to use the rankings of the samples.  It
should also be mentioned that such ranking correlations are designed
to detect monotone relationships, which --- although more general than
the linearity of Pearson's correlation --- is still a significant
restriction. There are multiple ways to measure the similarity of the
orderings of the data when ranked by each of the quantities. The most
common is the Spearman rank correlation method, which is the Pearson
correlation of the ranks. Alternatives are Kendall's $\tau$ and
Goodman and Kruskall's $\gamma$, which measure pairwise
concordance. We say that a pair of samples is \emph{concordant} if the
cases are ranked in the same order for both variables. They are
\emph{reversed} if the orders differ. We drop any pair of samples
where the values in either of the variables is equal. We then define
\begin{equation*}
  \tau:=\frac{N_{\mathrm{s}}-N_{\mathrm{d}}}{N_{\mathrm{s}}+N_{\mathrm{d}}}\qquad \text{and} \qquad  \gamma:=\frac{N_{\mathrm{s}}-N_{\mathrm{d}}}{n(n-1)/2},
\end{equation*}
where $N_{\mathrm{s}}$ is the number of concordant pairs,
$N_{\mathrm{d}}$ is the number of reversed pairs, and $n$ is the total
number of samples. The only difference between these rank correlations
is the treatment of pairs with equal rank; $\tau$ penalises ties while
$\gamma$ does not. While these methods are distribution free, they are
not suitable to be directly applied to topological summaries as the
summaries do not lie in spaces with an order. We cannot rank the
samples, and thus we can not apply tests that use the ranks.

A new, non-parametric alternative is to work with the pairwise
distances. Given paired samples $(X,Y) = \{(x_i,y_i) \mid i =
1,...,n\}$, where the $x_i$ and $y_i$ lie in metric spaces
$\mathcal{X}$ and $\mathcal{Y}$, respectively, we can ask what the
joint variability of the pairwise distances is (\ie/ how related
$d_\mathcal{Y}(y_i, y_j)$ is to $d_\mathcal{X}(x_i, x_j)$). The
statistical tools of distance covariance and distance correlation are
apt for this purpose. The notion was introduced in \cite{SRB} for the
case of samples lying in Euclidean space. 

%They showed that it can test
%independence, with a distance correlation of zero if and only if the
%variables are independent \fixme{I don't understand the sentence. I've
%  added the comma and the indefinite article in front of ``distance
%  correlation'' on a hunch, but I'm afraid that I may have changed the
%  meaning of the sentence.}. The method makes no assumptions on the
%distributions and can detect relationships that are highly
%nonlinear. In contrast, we can only conclude from zero Pearson
%correlation coefficient that the variables are independent when we can
%assume that the joint distribution is bivariate normal. This
%difference between Pearson correlation and distance correlation is
%illustrated with real valued random variables in
%Figure~\ref{fig:PearsonvsDistance}.
%
%Distance correlation can be applied to distributions of samples lying
%in more general metric spaces. It can detect relationships between
%variables that are not linear, and not even monotonic. If the
%variables are independent, then the distance correlation is zero. In
%the other direction, if the metrics spaces are of strong negative
%type, then distance correlation of zero implies that the variables are
%independent. This is true for any joint distribution.

Distance correlation can be applied to distributions of samples lying
in more general metric spaces. It can detect relationships between
variables that are not linear, and not even monotonic, as can be seen
in Figure~\ref{fig:PearsonvsDistance}. There are strong theoretical
results about independence. If the variables are independent then the
distance correlation is zero. In the other direction, if the metrics
spaces are separable and of strong negative type then distance
correlation of zero implies the variables are independent. This is
discussed in more detail in section~\ref{sec:Negtype}.

In contrast, we can only conclude from Pearson correlation coefficient
being zero that the variables are independent when we can assume that
the joint distribution is bivariate normal. This difference between
Pearson correlation and distance correlation is illustrated with real
valued random variables in Figure~\ref{fig:PearsonvsDistance}.

The formal definitions follow.

\begin{figure}[htbp]
  \centering
  \includegraphics[width=0.8\linewidth]{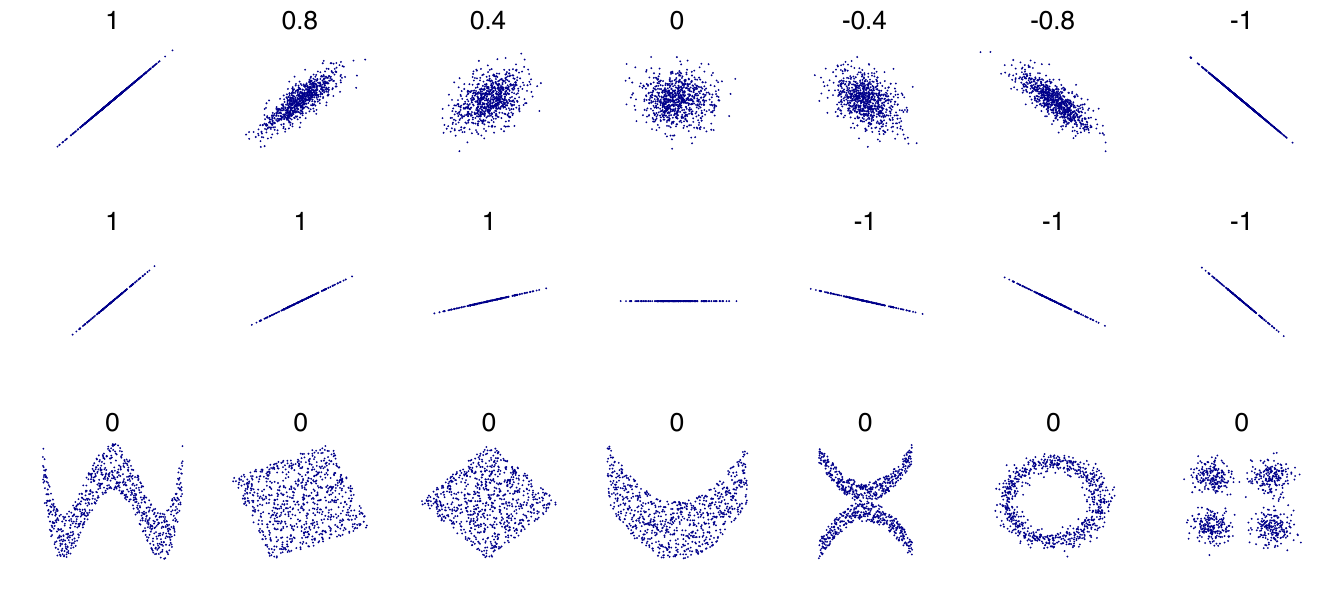}\\
  \vspace{3cm}
  \includegraphics[width=0.8\linewidth]{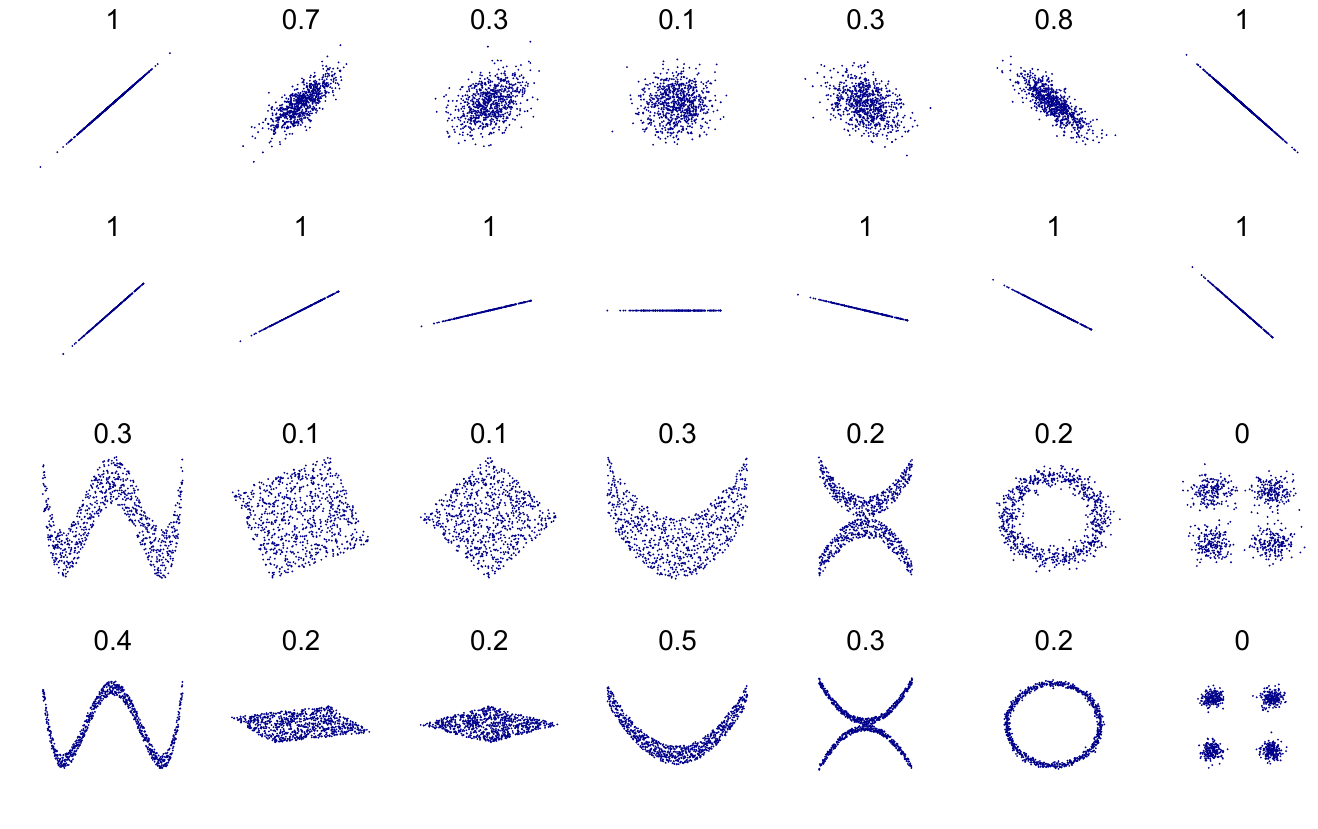}
  \caption{These examples demonstrate how distance correlation
    (\textbf{top}) is much more useful than Pearson correlation
    (\textbf{bottom}) when the joint distribution is not a
    multivariate Gaussian. The figures are taken
    from~\cite{wiki-corr,wiki-distcorr}.}
  \label{fig:PearsonvsDistance}
\end{figure}

\begin{defn}
Let $X$ be a random element taking values in a connected metric space
$(\mathcal{X}, d_{\mathcal{X}})$ with distribution $\mu$. For $x\in
\mathcal{X}$ we call $\E[d_{\mathcal{X}}(x,X)]$ the \emph{expected
  distance} of $X$ to $x$, and denote it by $a_\mu(x)$. We say that
$X$ has \emph{finite first moment} if for any $x\in\mathcal{X}$ the
expected distance to $x$ is finite. In this case we set
$D(\mu):=\E[a_\mu(x)]$. For $X$ with finite first moment we define its
\emph{doubly centred distance function} as
\begin{equation*}
  d_{\mu}(x,x'):=d_{\mathcal{X}}(x,x')-a_\mu(x) -a_\mu(x')+D(\mu).
\end{equation*}
\end{defn}

It is worth observing that $d_\mu$ is not a distance function. Lyons
showed~\cite{Lyons} that $a_\mu(x)>D(\mu)/2$ for all $x$ as long as
the support of $\mu$ contains at least two points. This implies
$d_\mu(x,x)<0$ for all $x$.

\begin{defn}
Let $\mathcal{X}$ and $\mathcal{Y}$ be metric spaces. Let
$\theta=(X,Y)$ be a probability distribution over the product space
$\mathcal{X} \times \mathcal{Y}$ with marginals $\mu$ and $\nu$ such
that $X$ and $Y$ both have finite first moment.  We define the
\emph{distance covariance} of $\theta$ as
\begin{equation*}
  \dcov(\theta)=\int d_\mu(x,x)d_\nu(y,y')d\theta^2((x,y),(x',y')).
\end{equation*}
\end{defn}

The \emph{distance variance} is a special case where have two
identical copies as the joint distributions $\theta^X=(X,X)$. Here we
have
\begin{equation*}
  \dvar(\theta^X)=\int d_\mu(x,x')^2d\theta^2((x,x')),
\end{equation*}
which is always non-negative, and zero only in the case of a
distribution supported on a single point.

The \emph{distance correlation} of $\theta=(X,Y)$ is defined as
\begin{equation*}
  \dcor(X,Y)=\frac{\dcov(X,Y)}{\sqrt{\dvar(\theta^X)\dvar(\theta^Y)}}.
\end{equation*}

\begin{remark}
There are some variations of notation with regard to whether to
include a square root in the definition of distance covariance and
correlation. In the introduction of distance correlation in
\cite{SRB}, the authors restricted their analysis to Euclidean spaces. Euclidean
spaces are metric spaces of negative type, and such spaces have the
property that the distance correlation is always non-negative. They
could thus define the distance covariance as
\begin{equation*}
 \sqrt{\int d_\mu(x,x)d_\nu(y,y')d\theta^2((x,y),(x',y'))}.
\end{equation*}
We will follow the notation of~\cite{Lyons} and use $\dCov$ to denote
the square root of $\dcov$, \ie/
\begin{equation*}
 \dCov(X,Y)=\sqrt{\dcov(X,Y)}=\sqrt{\int d_\mu(x,x)d_\nu(y,y')d\theta^2((x,y),(x',y'))}.
\end{equation*}
We will also use $\dVar$ as the square root of the distance variation
and $\dCor$ to denote the square root of the the distance correlation, \ie/
\begin{equation*}
  \dCor(X,Y)=\frac{\dCov(X,Y)}{\sqrt{\dVar(X)\dVar(Y)}}=\sqrt{\dcor(X,Y)}.
\end{equation*}
In the simulations and calculations involving topological summaries,
it turns out that all the values of the distance correlation are
non-negative, even for those involving spaces that are not of negative type.
\end{remark}

Given a set of paired samples drawn from a joint distribution, we can
compute a \emph{sample distance covariance}. This is an estimator of
the distance covariance of a joint distribution from which the paired
samples was taken.

The estimation of the distance correlation of a joint distribution by
sample distance covariances is reasonable. In other words this means
that if $\theta_n$ is the sampled joint distribution from $n$
i.i.d.\ samples of $\theta$ then $\dcov(\theta_n)\rightarrow
\dcov(\theta)$ as $n\to \infty$ with probability $1$. See Proposition
2.6 in \cite{Lyons}. This justifies the approximation of the distance
correlation via simulations. This is particularly important when
dealing with distributions for which there is no closed expressions,
which is usually the case when dealing with topological summaries

The following procedure computes the sample distance covariance
between paired samples $(X,Y) = \{(x_i,y_i) \mid i = 1,...,n\}$, which
we denote $\dcov_n (X,Y)$:
\begin{enumerate}
\item Compute the pairwise distance matrices $a=(a_{i,j})_{i,j}$,
  $b=(b_{i,j})_{i,j}$ with $a_{i,j}=d_{\mathcal{X}}(x_i, x_j)$ and
  $b_{i,j}=d_{\mathcal{Y}}(y_i,y_j)$.
\item Compute the means of each row and column in $a$ and $b$ as well
  as the total means of the matrices. Let $\bar{a}^i$ and $\bar{b}^i$
  denote the row means and $\bar{a}_j$ and $\bar{b}_j$ the column
  means. Let $\bar{a}$ and $\bar{b}$ denote the total matrix means.
\item Compute doubly centered matrices $(A_{k,l})_{k,l}$ and
  $(B_{k,l})_{k,l}$ with $A_{k,l}=a_{k,l}-\bar{a}^k -\bar{a}_l
  +\bar{a}$ and $B_{k,l}=b_{k,l}-\bar{b}^k -\bar{b}_l +\bar{b}$
\item The sample distance covariance is
  \begin{equation*}
    \dcov_n=\frac{1}{n^2}\sum_{k,l=1}^n A_{k,l} B_{k,l}
  \end{equation*}
\end{enumerate}

Note that the matrices $A$ and $B$ have the property that all rows and
columns sum to zero.

\subsection{Metric spaces of strong negative type}\label{sec:Negtype}

As straightforward application of the definition shows that the
distance correlation of a product measure is always zero. To see this,
observe that when $\theta$ is a product of $\theta_X$ and $\theta_Y$,
then
\begin{align*}
  \dcov(\theta) &= \int d_\mu(x,x')d_\nu(y,y')d\theta^2((x,y),(x',y')) \\
  &= \int d_\mu(x,x') d\theta^2_X(x,x') \int d_\nu(y,y')d\theta^2_Y(y,y').
\end{align*}
By construction of $d_\mu$ and $d_\nu$, we have $\int d_\mu(x,x')
d\theta^2_X(x,x') =0=\int d_\nu(y,y')d\theta^2_Y(y,y')$. The converse
of this statement holds under conditions on the metric spaces the
distributions are over (not the distributions themselves).

\begin{defn}
  A metric space $(X,d)$ has \emph{negative type} if for all $x_1,
  \ldots, x_n\in X$ and $\alpha_1, \ldots, \alpha_n\in\mathbb{R}$ with
  $\sum_i \alpha_i=0$
  \begin{equation}
    \sum_{i,j=1}^n \alpha_i\alpha_j d(x_i,x_j)\leq 0. \label{eq:negtype}
  \end{equation}
\end{defn}

%Note that the definition of negative type coincides with that of
%conditionally negative semidefinite. \fixme{Should we refer to
 % wherever this is defined?}

%It was shown in Schoenberg (1937, 1938) that an equivalent definition of negative type is the existence of a Hilbert space embedded $\phi:X\to \mathcal{H}$ with $\|\phi(x)-\phi(x')\|^2=d(x,x')$.

For spaces of negative type it is always true that the distance
covariance is non-negative~\cite{Lyons}. We have further nice
properties when the metric space is of strong negative type.

\begin{defn}
  A metric space has \emph{strict negative type} if it is a space of
  negative type where equality in~\eqref{eq:negtype} implies that the
  $\alpha_i$ are all zero. By extending to distributions of infinite
  support we get the definition of strong negative type: A metric
  space $(\mathcal{X}, d)$ has \emph{strong negative type} if it has
  negative type and for all probability measures $\mu_1, \mu_2$ we
  have
  \begin{equation*}
    \int d(x,x')d(\mu_1-\mu_2)^2(x,x')\leq 0.
  \end{equation*}
\end{defn}

Lyons~\cite{Lyons} used the notion to characterize the spaces where
one can test for independence of random variables using distance
correlation.

%\fixme{There needs to be more meat surrounding this theorem, it pops
  %out of nowhere.}

\begin{theorem}[\citeauthor{Lyons}~\cite{Lyons}, \citeyear{Lyons}]
Suppose that $\mathcal{X}$ and $\mathcal{Y}$ are separable metric spaces of strong negative type
and that $\theta$ is a probability measure on $\mathcal{X} \times
\mathcal{Y}$ whose marginals have finite first moment. If
$\dcov(\theta)=0$, then $\theta$ is a product measure.
\end{theorem} 

This means that given paired random variable $(X,Y)$ with joint
distribution $\theta$, we can test for independence by computing
$\dcov(\theta)$ and decide they are independent if $\dcov(\theta)=0$,
and not independent if $\dcov(\theta)>0$. The challenge is then how to
implement such a test given a sample distance correlation. We expect
the sample distance correlation to be non-zero even when the variables
are independent.

There is a range of spaces that are proven to be of strong negative
type, including all separable Hilbert spaces.

\begin{theorem}[\citeauthor{Lyons}~\cite{Lyons}, \citeyear{Lyons}] \label{thm:SepHilbert}
Every separable Hilbert space is of strong negative type. Moreover, if
$(X,d)$ has negative type, then $(X,d^r)$ has strong negative type
when $0 < r < 1$.
\end{theorem}

A list of metric spaces of negative type appears as Theorem 3.6
of~\cite{Meckes}; in particular, this includes all $L^p$ spaces for $1
\leq p \leq 2$. On the other hand, $\R^n$ with the $l^p$-metric is not
of negative type whenever $3 < n <\infty$ and $2 < p < \infty$.

%The distance correlation is non-negative for spaces of negative type.

The distance correlation still contains useful information even when
the spaces are not of strong negative type. It is just more powerful
as a test statistic when the spaces are of strong negative type. This
is analogous to how the Pearson correlation coefficient still can be
evidence of a relationship between two variables even when the joint
distribution is not Gaussian. Here the Pearson correlation coefficient
is detecting linear relationships. It is an open problem to
characterise which relationships are, and which are not, detected by
the distance correlation in spaces that are not of strong negative
type.

Distance correlation lends itself to non-parametric methods. One
possibility is to combine it with permutation tests to construct
$p$-values for independence. Permutation tests construct a sampling
distribution by resampling the observed data. We can permute the
observed data without replacement to create a null distribution (in
this case a distribution of distance correlation values under the
assumption that the random variables are independent). The use and
exploration of permutation tests in relation to distance correlation
is beyond the scope of this paper. We direct the interested reader to
section~\ref{sec:future} for more details.

\section{A veritable zoo of topological summaries, some of which are of strong negative type}

Persistent homology has become a very important tool in TDA. Certainly
there are many choices that are made in any persistent homology
analysis, with much of the focus being on the filtration. In this
paper we want to highlight another choice, namely the metric space
structure to put on the topological summary of choice. Examples
include persistence diagrams with bottleneck and Wasserstein
distances, persistence landscapes or rank function with an $L^p$
distance, or one of the many kernel representations. The choice of
which topological summary we use to represent persistent homology, and
the choice of metric on this space of topological summaries, will
affect any statistical analysis and will influence whether or not the
summary captures the information that is of relevance to the
application.

For spaces of strong negative type, distance correlation is known to
have the additional nice properties. As a rule, functional spaces with
an $L^2$ metric and those lying in a reproducing kernel Hilbert space
are of strongly negative type. This implies that the Euler
characteristic and Betti curves with the $L^2$ metric are of strong
negative type, and that the space of persistence scale shape kernels
is of strong negative type. In this section we will characterise which
of the spaces of persistence landscapes are of strong negative type
and show that the space of persistence diagrams is never of strong
negative type. The main results are as follows.

\begin{theorem*}[Theorem \ref{thm:PD}]
	The space of persistence diagrams is not of negative type under the bottleneck metric or under any of the Wasserstein metrics.
\end{theorem*}

\begin{theorem*}[Theorem \ref{thm:Landscapes}]
	\begin{enumerate}
		\item[(a)] The space of persistence landscapes with the $L^2$ norm is of strong negative type.
		\item[(b)] The space of persistence landscapes with the $L^p$ norm is of negative type when $1\leq p\leq 2$ 
		\item[(c)] The space of persistence landscapes with the $L^1$ norm is not of strong negative type, even when restricting to persistence landscapes that arise from persistence diagrams.
		\item[(d)] The space of persistence landscapes with the $L^\infty$ norm is not of negative type, even when restricting to persistence landscapes that arise from persistence diagrams.
	\end{enumerate}
\end{theorem*}

It is an open question as to whether the sliced Wasserstein metric is of strong negative type; if it is separable then it will be.

%How all the function spaces with $p=1$ and $p=2$ are of strong negative type. (With token construction outline of the embedding)

%``It was proved by Dor [4] that if $2 < p \leq \infty$ then $l_3^p$ does not embed isometrically in $L_1$, and is therefore not positive definite. In particular, $L_p$ is not positive definite for any $p > 2$.''

%Topological summaries in the form of curves have been studied for a long time. In particular, the Betti curves and the Euler characteristic curves have been ....
%There has been an increasing amount of interest in functional topological summaries that are built using all of the persistent homology information. Much of the interest stems from the ease of applying standard techniques from statistics and computer science. 

%Since  every separable Hilbert space is of strong negative type. Thus the $L^2$ versions of metric spaces for landscapes, rank functions and heat maps are all of strong negative type.

%The $L^1$ versions are of negative type but not of strong negative type ....
%However, $L^p$ versions are for all $1<p<2$. We could also consider the $d_1^r$ for any $r\in[0,1]$ which will be of strong negative type.

%Functional topological summaries using $L^\infty$ metrics are not of negative type. We will be focusing on the $p=1$ and $p=2$ cases

\subsection{Betti and Euler characteristic curves}
Some of the first topological summaries often considered for
parameterised families of topological spaces ($\{K_a\}$) are the Betti
and the Euler characteristic curves, which we denote by
$\beta_k:\mathbb{R}\to\mathbb{N}_0$ and
$\chi:\mathbb{R}\to\mathbb{Z}$. These are integer valued functions
with $\beta_k(a)=\dim H_k(K_a)$ and $\chi(a)=\chi(K_a)$. From the
point of view of barcodes, one thinks of $\beta_k(a)$ as the number of
bars that contain the point $a$. The Euler curve is then the
alternating sum of the Betti curves, $\chi(a)=\sum_{k=0}^\infty (-1)^k
\beta_k(a)$, as one would expect.

Clearly the Betti and Euler curves contain less information than the
persistence diagrams; in particular, the Betti curves can be thought
of as encoding point-wise homological information without considering the induced
maps $H_k(a)\to H_k(b)$.

Since $\beta_k$ and $\chi$ are functions, we can consider functional
distances between them. In this paper we consider both $L^1$ and $L^2$
distances. Since $L^2(\R)$ is a separable Hilbert space, it is of
strong negative type. In comparison $L^1(\R)$ is of negative type, but
not of strict negative type (see~\cite{Lyons}). For an explicit
counterexample, the reader can modify the one used for the $p=1$ case
in section~\ref{sec:landscapes}.

\subsection{Persistence diagrams}
Persistence diagrams are arguably the most common way of representing
persistent homology. A persistence diagrams is a multiset of points
above the diagonal in the real plane, with lines at $\pm\infty$ in the
second coordinate. 
%\fixme{I removed the stuff about finitely many
  %infinite bars and the sum of all finite lifetimes is finite; we
  %don't need it, and we save some explaining this way, right?} Much of
%what follows holds also under more relaxed conditions relating to
%finiteness or tameness.

%% For tractability, we will impose some finiteness
%% conditions, namely that only finitely many classes have infinite
%% lifetimes and that the sum of all the finite lifetimes is finite.
%This restriction is not onerous in applications where generally we
%have finite sized data as input.
%Notably under all of these metrics $\L_\infty \cup \L_{-\infty}\cup \R^{2+}$ has connected components $\L_\infty$, $\L_{-\infty}$ and $\R^{2+}$.
In what follows, let $\R^{2+}=\{(x,y)\in \R^2\mid x<y\}$ be the subset
of the plane above the diagonal $\Delta=\{(x,x) \mid x \in\R \}$, and
let $\L_{\pm\infty} = \{(x,\pm\infty) \mid x\in\R\}$ denote horizontal
lines at infinity.

\begin{defn}
A \emph{persistence diagram} $X$ is a multiset in $\L_\infty \cup
\L_{-\infty}\cup \R^{2+}\cup \Delta$ such that
\begin{itemize}
\item The number of elements in $X|_{\L_{\infty}}$ and $X|_{\L_{-\infty}}$ are finite
\item $\sum_{(x_i,y_i)\in X\cap{\R^{2+}}}(y_i-x_i)<\infty$ 
  %\fixme{Whyshould we be this technical when we're restricting everything to the finite case anyway?}
%\item There are countably infinite copies of an abstract element representing the diagonal in the plane which we denote by $\Delta$ 
  %\fixme{This is superfluous given what's already written above.}
\item $X$ contains countably infinite copies of $\Delta$.
\end{itemize}
\end{defn}

For our purposes, it suffices to consider persistence diagrams with
only finitely many off-diagonal points.

Let $\D$ denote the set of all persistence diagrams. We will consider
a family of metrics which are analogous to the $p$-Wasserstein
distances on the set of probability measures, and to the $L^p$
distances on the set of functions on a discrete set. $\R^{2+}$
inherits natural $L^p$ distances from $\R^2$. For $p\in [1,\infty)$ we
  have $\|(a_1, b_1) - (a_2, b_2)\|_p^p = |a_1-a_2|^p + |b_1-b_2|^p$
  and $\|(a_1, b_1) -(a_2, b_2)\|_\infty = \max \{ |a_1-a_2|, |b_1-
  b_2|\}$.

With a slight abuse of notation we write $\|(a, b) - \Delta\|_p $ to
denote the shortest $L^p$ distance to $\Delta$ from a point $(a,b)$ in
a persistence diagram. Thus
\begin{equation*}
  \|(a, b) - \Delta\|_p = \inf_{t\in \R}\|(a,b)-(t,t)\|_p= 2^{\frac{1}{p}-1}|b-a|
\end{equation*}
for $p<\infty$, and $\|(a, b)-\Delta\|_\infty=\inf_{t\in
  \R}\|(a,b)-(t,t)\|_\infty = |y-x|/2$. Both $ \L_{-\infty}$ and
$\L_\infty$ inherit natural $L^p$ distances from the $L^p$ metric on
$\R$, \ie/ $\|(-\infty, b_1)-(-\infty, b_2)\|_p=|b_1-b_2|$ and
$\|(a_1,\infty)-(a_2,\infty)\|_p=|a_1-a_2|$.

Given persistence diagrams $X$ and $Y$, we can consider all the
bijections between them. This set is non-empty due to the presence of
$\Delta$ in the diagrams. Each bijection can be thought of as
providing a transport plan from $X$ to $Y$. One defines a family of
metrics in terms of the cost of the most efficient transport plan.

For each $p \in [1, \infty)$, define
\begin{equation*}
  d_p(X,Y) = \left( \inf_{\substack{\phi:X \to Y\\\text{bijection}}}\, \sum_{x\in X} \|x-\phi(x)\|_p^p\right)^{1/p}
\end{equation*}
and
\begin{equation*}
  d_\infty(X,Y) = \inf_{\substack{\phi:X \to Y\\\text{bijection}}}\, \sup_{x\in X} \|x-\phi(x)\|_\infty. 
\end{equation*}
These distances may be infinite. Indeed, if $X$ and $Y$ contain a
different number of points in $\mathcal{L}_\infty$, then
$d_p(X,Y)=\infty$ for all $p$.

In theory, for every pair $p,q\in [1,\infty]$ one can construct a
distance function of the form
\begin{equation*}
  \inf_{\phi:X \to Y} \left(\sum_{x\in X} \|x-\phi(x)\|_q^p\right)^{1/p}
\end{equation*}
with $p$ and $q$ potentially different. Some of the computational
topology literature uses a family of metrics $d_{W_p}$ where $p$
varies but $q=\infty$ is fixed.
%This choice is mainly for historical reasons. 
%The first and most prominent metric on the space of diagrams is the bottleneck distances  $d_\infty(X, Y)=\inf_{\phi} \max_{x\in X}\|x-\phi(x)\|_\infty$ and the first generalizations retained the $q= \infty$ rather than making $q=p$. 
The families $\{d_p\}$ and $\{d_{W_p}\}$ share many properties. The
metrics $d_p$ and $d_{W_p}$ are bi-Lipschitz equivalent, as for any
$x,y \in \R^2$ we have $\|x-y \|_\infty \leq \|x -y \|_p \leq 2 \|x -
y \|_\infty$, implying $d_{W_p}(X,Y)\leq d_{p}(X,Y) \leq 2
d_{W_p}(X,Y)$. Any stability results (\ie/ results pertaining to the
change in persistence diagrams due to perturbations of the underlying
filtration) for $\{d_p\}$ or $\{d_{W_p}\}$ extend (with minor changes
in the constants involved) to stability results for the other.

We feel that the choice of $q=p$ is cleaner in theory and in
practice. The coordinates of the points within a persistence diagram
have particular meanings; one is the birth time and one is the death
time. They are often locally independent (even though not globally
so). For example, if we have generated our persistence diagram from
the distance function to a point cloud, then each persistence class
has its birth and death time locally determined by the location of two
pairs of points, which are often distinct. Whenever these pairs are
distinct, moving any of these four points will change either the birth
or the death but not both. The distinctness of the treatment of birth
and death times as separate qualities may seem more philosophically
pleasing to the reader in the setting of barcodes.

Unfortunately, the geometry of the space of persistence diagrams is
complicated and statistical methods not easy to apply. For example,
there are challenges even in computing the mean or median of finite
samples (see~\cite{mean,median}). Given this it is perhaps not
surprising that the space of persistence diagrams is not of negative
type (let alone of strong negative type) under the bottleneck or
indeed any of the Wasserstein metrics. Although this has been
indirectly mentioned or suggested before (notably
in~\cite{Heat,Sliced}), we include here explicit counterexamples.

\begin{thm}\label{thm:PD}
	The space of persistence diagrams is not of negative type under the bottleneck or any of the Wasserstein metrics.
\end{thm}

\begin{proof}

We will construct two different counterexamples; one for small $p$ and
one for large $p$. Note that the bottleneck metric is the Wasserstein
metric with $p=\infty$.

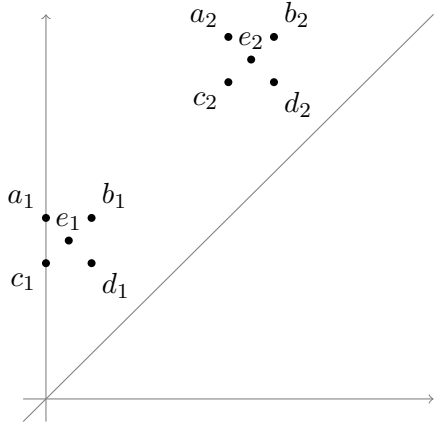
\begin{figure}[htbp]
\begin{center}
\begin{tikzpicture}[scale=.3]
\draw[gray, ->] (0,-1)--(0,17);
\draw[gray, ->](-1,0)--(17,0);
\draw[gray](-1,-1)--(17,17);

\fill[] (0, 6) circle (5pt) node[below left] {$c_1$};
\fill[] (2, 6) circle (5pt) node[below right] {$d_1$};
\fill[] (0, 8) circle (5pt) node[above left] {$a_1$};
\fill[] (2, 8) circle (5pt) node[above right] {$b_1$};
\fill[](1, 7) circle(5pt) node[above] {$e_1$};
\fill[] (8, 14) circle (5pt) node[below left] {$c_2$};
\fill[] (10, 14) circle (5pt) node[below right] {$d_2$};
\fill[] (8, 16) circle (5pt) node[above left] {$a_2$};
\fill[] (10, 16) circle (5pt) node[above right] {$b_2$};
\fill[](9, 15) circle(5pt) node[above] {$e_2$};

\end{tikzpicture}
\caption{The off-diagonal points used in the persistence diagrams in
  the counterexample for $p\leq 2.4$.} \label{fig:ptsexample}
\end{center}
\end{figure}

For small $p$, consider the two separate unit squares formed by the
points $a_1,b_1,c_1,d_1$ and $a_2,b_2,c_2,d_2$ in
Figure~\ref{fig:ptsexample}. Each persistence diagram will be a union
of a pair of corners sharing an edge in one of the squares, together
with a pair of corners diagonally opposite on the other square. We
then choose the weights (the $\alpha$'s in
inequality~\eqref{eq:negtype}) to be $1$ if the off-diagonal points
are diagonally opposite in the rightmost square, and $-1$ if they are
diagonally opposite in the leftmost square. A list of the diagrams is
in Table~\ref{tab:smallp}.
\begin{table}[htbp]
  \begin{center}
    \begin{tabular}{c|c|c}
      Diagram & Off-diagonal points & Weight\\
      \hline
      $x_1$ & $\{a_1,b_1, a_2,d_2\}$ & $1$\\
      $x_2$ & $\{a_1,c_1, a_2,d_2\}$ & $1$\\
      $x_3$ & $\{b_1,d_1, a_2,d_2\}$ & $1$\\
      $x_4$ & $\{c_1,d_1, a_2,d_2\}$ & $1$\\
      $x_5$ & $\{a_1,b_1, b_2,c_2\}$ & $1$\\
      $x_6$ & $\{a_1,c_1, b_2,c_2\}$ & $1$\\
      $x_7$ & $\{b_1,d_1, b_2,c_2\}$ & $1$\\
      $x_8$ & $\{c_1,d_1, b_2,c_2\}$ & $1$\\
    \end{tabular}
    \quad
    \begin{tabular}{c|c|c}
      Diagram & Off-diagonal points & Weight\\
      \hline
      $y_1$ & $\{a_1,d_1, a_2,b_2\}$ & $-1$\\
      $y_2$ & $\{a_1,d_1, a_2,c_2\}$ & $-1$\\
      $y_3$ & $\{a_1,d_1, b_2,d_2\}$ & $-1$\\
      $y_4$ & $\{a_1,d_1, c_2,d_2\}$ & $-1$\\
      $y_5$ & $\{b_1,c_1, a_2,b_2\}$ & $-1$\\
      $y_6$ & $\{b_1,c_1, a_2,c_2\}$ & $-1$\\
      $y_7$ & $\{b_1,c_1, b_2,d_2\}$ & $-1$\\
      $y_8$ & $\{b_1,c_1, c_2,d_2\}$ & $-1$\\
    \end{tabular}
    \caption{Counterexample showing that the space of persistence
      diagrams with $W_p$ is not of negative when
      $p<\ln(2)/\ln(4/3)$. The off-diagonal points come from
      Figure~\ref{fig:ptsexample}. The \textit{weight} columns refer
      to the $\alpha$'s in
      inequality~\eqref{eq:negtype}.} \label{tab:smallp}
  \end{center}
\end{table}

We have the following distance matrix for the within-group distances,
\ie/ the symmetric matrix with entries $(d_p(x_i,
x_j))_{i,j}=(d_p(y_i,y_j))_{i,j}$:
\begin{equation*}
  \begin{bmatrix}
    0       	& 2^{1/p} 	& 2^{1/p} & 2^{1/p} & 2^{1/p} & 4^{1/p} & 4^{1/p} &4^{1/p} \\
    2^{1/p} &   0   & 2^{1/p} & 2^{1/p} & 4^{1/p} & 2^{1/p} &4^{1/p} & 4^{1/p} \\
    2^{1/p} & 2^{1/p}  & 0   & 2^{1/p} & 4^{1/p} & 4^{1/p} & 2^{1/p} &4^{1/p} \\
    2^{1/p} & 2^{1/p}     & 2^{1/p} & 0  &4^{1/p} & 4^{1/p} &4^{1/p} & 2^{1/p} \\
    2^{1/p}       	& 4^{1/p} 	& 4^{1/p} & 4^{1/p} & 0 & 2^{1/p} & 2^{1/p} &2^{1/p} \\
    4^{1/p} &   2^{1/p}   & 4^{1/p} & 4^{1/p} & 2^{1/p} & 0 &2^{1/p} & 2^{1/p} \\
    4^{1/p} & 4^{1/p}  & 2^{1/p}   & 4^{1/p} & 2^{1/p} & 2^{1/p} & 0 &2^{1/p} \\
    4^{1/p} & 4^{1/p}     & 4^{1/p} & 2^{1/p}  &2^{1/p} & 2^{1/p} &2^{1/p} & 0 \\
  \end{bmatrix}
\end{equation*}
This implies that
\begin{equation*}
  \sum_{i,j=1}^8 d_p(x_i, x_j) = \sum_{i,j=1}^8 d_p(y_i, y_j) = 32 \cdot 2^{1/p} + 24 \cdot 4^{1/p},
\end{equation*}
and similarly
\begin{equation*}
  \sum_{i,j=1}^8 d_p(x_i,y_j)=\sum_{i,j=1}^8 d_p(y_i, x_j)= 64 \cdot 2^{1/p}.
\end{equation*}
The sum of interest, using the weighting in Table~\ref{tab:smallp}, is
\begin{align*}
  \sum_{i,j=1}^8 d_p(x_i, x_j) &+\sum_{i,j=1}^8 d_p(y_i, y_j) -\sum_{i,j=1}^8 d_p(x_i,y_j)-\sum_{i,j=1}^8 d_p(y_i, x_j)\\
  &= 64 \cdot 2^{1/p} + 48 \cdot 4^{1/p} -128 \cdot 2^{1/p}\\
  & =48 \cdot 4^{1/p}-64 \cdot 2^{1/p}. 
\end{align*}
Now $48 \cdot 4^{1/p}-64 \cdot 2^{1/p}>0$ exactly when
$p<\ln(2)/\ln(4/3)$. This thus shows that the metric space of
persistence diagrams with $W_p$ is not of negative type when
$p<\ln(2)/\ln(4/3)$.

We will now construct a counterexample for space of persistence diagrams under $p$-Wasserstein distance with $p\geq 2.4$. We will construct our counterexample with persistence diagrams containing points listed in Figure~\ref{fig:ptsexample2}. This has separate squares with unit edge length that are sufficiently far apart.
% as displayed in Figure \ref{fig:ptsexample2}.
We will have two sets of persistence diagrams, $X$ and $Y$, and we will be giving a weight of $1$ to all the persistence diagrams in $X$ and a weight of $-1$ to all the persistence diagrams in $Y$.

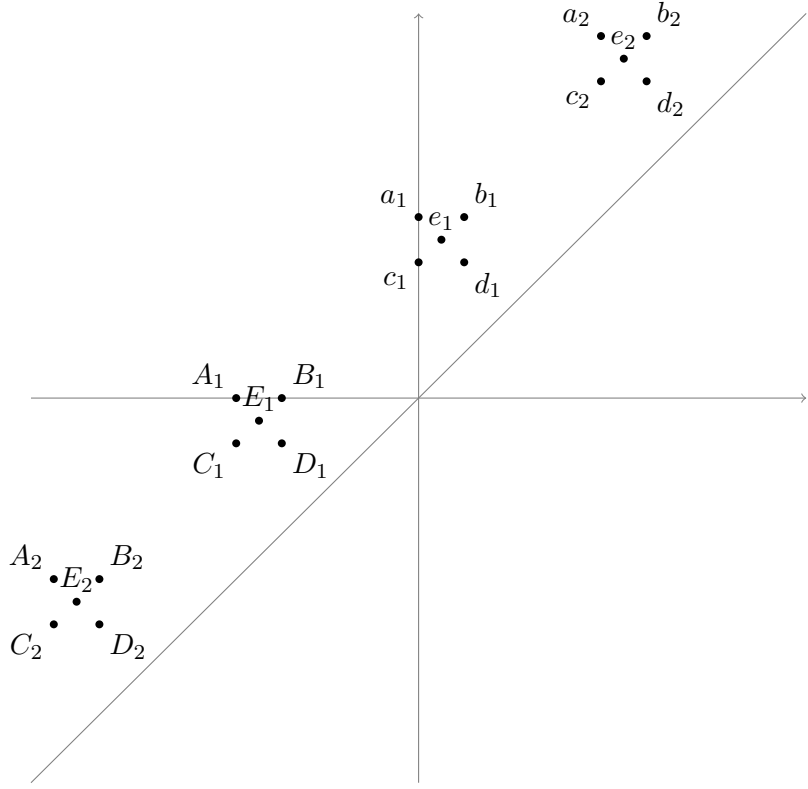
\begin{figure}[htbp]
	\begin{center}
		\begin{tikzpicture}[scale=.3]
		\draw[gray, ->] (0,-17)--(0,17);
		\draw[gray, ->](-17,0)--(17,0);
		\draw[gray](-17,-17)--(17,17);
		
		\fill[] (0, 6) circle (5pt) node[below left] {$c_1$};
		\fill[] (2, 6) circle (5pt) node[below right] {$d_1$};
		\fill[] (0, 8) circle (5pt) node[above left] {$a_1$};
		\fill[] (2, 8) circle (5pt) node[above right] {$b_1$};
		\fill[](1, 7) circle(5pt) node[above] {$e_1$};
		\fill[] (8, 14) circle (5pt) node[below left] {$c_2$};
		\fill[] (10, 14) circle (5pt) node[below right] {$d_2$};
		\fill[] (8, 16) circle (5pt) node[above left] {$a_2$};
		\fill[] (10, 16) circle (5pt) node[above right] {$b_2$};
		\fill[](9, 15) circle(5pt) node[above] {$e_2$};
		
		\fill[] (-8, -2) circle (5pt) node[below left] {$C_1$};
		\fill[] (-6, -2) circle (5pt) node[below right] {$D_1$};
		\fill[] (-8, 0) circle (5pt) node[above left] {$A_1$};
		\fill[] (-6, -0) circle (5pt) node[above right] {$B_1$};
		\fill[](-7, -1) circle(5pt) node[above] {$E_1$};
		\fill[] (-16, -10) circle (5pt) node[below left] {$C_2$};
		\fill[] (-14, -10) circle (5pt) node[below right] {$D_2$};
		\fill[] (-16, -8) circle (5pt) node[above left] {$A_2$};
		\fill[] (-14, -8) circle (5pt) node[above right] {$B_2$};
		\fill[](-15, -9) circle(5pt) node[above] {$E_2$};

		\end{tikzpicture}
		\caption{The off-diagonal points used in the persistence diagrams in the counterexamples for $p\geq 2.4$.}\label{fig:ptsexample2}
	\end{center}
\end{figure}

Each persistence diagram in $X$ will have 4 off-diagonal points; one corner point from each of the squares labelled with upper case letters, and $e_1$ and $e_2$. An example is $\{A_1, B_2, e_1, e_2\}$. There are a total of 16 such persistence diagrams.

Each persistence diagram in $Y$ will have 4 off-diagonal points; one corner point from each of the squares labelled with lower case letters, and $E_1$ and $E_2$. An example is  $\{c_1, c_2, E_1, E_2\}$.

For every pair of persistence diagrams $(x,y)\in X\times Y$ we have $d_p(x,y)=8^{1/p}/2$. This implies that the total between-group pairwise distance is $32\cdot16\cdot 8^{1/p}/2$.

To compute the within group distances we first observe that the symmetry of the counterexample ensures that the sum of distance $\sum_{x\in X}d_p(x,x')$ is the same for all $x'\in X$ and that this is also the same as $\sum_{y\in Y}d_p(y,y')$ for all $y'\in Y$. This means we can compute for a fixed $x'\in X$. We can split the remaining $x\in X$ into cases depending on how many of the off-diagonal points in the persistence diagrams are the same as that in $x'$, are on the same edge of the corresponding square as that in $x'$, or are diagonally opposite corners of the corresponding square. We describe this distribution in Table~\ref{tab:largep}, giving example persistence diagrams. 

\begin{table}[htbp]
  \begin{center}
    \begin{tabular}{c|c|c|c|c|c}
      Same        & Share an      & Diag.\ opp.          & Example      & No.\ of such           & \quad \\
      corner      & edge          & corners              & diagram      & $x\in X$               & $d_p(x,x')$ \\
%      Same corner & Share an edge & Diag.\ opp.\ corners & Ex.\ diagram & No.\ of such $x\in X$ & $d_p(x,x')$\\
      \hline
      $2$ & $0$  & $0$ & $\{A_1,A_2, e_1,e_2\}$& $1$ & $0$ \\
      $1$ & $1$ & $0$ & $\{A_1,B_2, e_1,e_2\}$ & $4$ & $1$ \\
      $1$ & $0$ & $1$ & $\{A_1,D_2, e_1,e_2\}$ & $2$ &  $2^{1/p}$\\
      $0$ & $2$ & $0$ &$\{B_1,C_2, e_1,e_2\}$  & $4$ &  $2^{1/p}$\\
      $0$ & $1$ & $1$ & $\{B_1,D_2, e_1,e_2\}$ & $4$ & $3^{1/p}$\\
      $0$ & $0$ & $2$ & $\{D_1,D_2, e_1,e_2\}$ & $1$ & $4^{1/p}$\\
    \end{tabular}
    \caption{Distances $d_p(x,x')$ for $x\in X$ and $x'=\{A_1, A_2,
      e_1, e_2\}$. The off-diagonal points are those shown in Figure
      \ref{fig:ptsexample}.}\label{tab:largep}
  \end{center}
\end{table}

Using this table, we calculate $\sum_{x\in X}d_p(x,x')=4+ 6\cdot2^{1/p}+ 4\cdot3^{1/p} + 4^{1/p}$. To prove this is a counterexample we need to show that 
$$32 \cdot (4+ 6\cdot2^{1/p}+ 4\cdot3^{1/p} + 4^{1/p})-32\cdot16\cdot 8^{1/p}/2>0.$$
This is equivalent to $4+ 6\cdot2^{1/p}+ 4\cdot3^{1/p} + 4^{1/p} > 8$ and by diving through by $8^{1/p}$ this is equivalent to the condition that
\begin{equation}
4(1/8)^{1/p}+ 6(1/4)^{1/p}+ 4(3/8)^{1/p} + (1/2)^{1/p}>8.\label{eq:suff}
\end{equation}

Now $\lambda^{1/p}$ is an increasing function in $p$, when $\lambda<1$ and $p>1$. Thus for all $p\geq 2.4$ we know $(1/8)^{1/p}\geq(1/8)^{1/2.4}>0.42$, $(1/4)^{1/p}\geq(1/4)^{1/2.4}>0.56$, $(3/8)^{1/p}\geq (3/8)^{1/2.4}>0.66$ and $(1/2)^{1/p}\geq (1/2)^{1/2.4}>0.74$. Together these imply that
$$4(1/8)^{1/p}+ 6(1/4)^{1/p}+ 4(3/8)^{1/p} + (1/2)^{1/p}>4\cdot 0.42+ 6\cdot 0.56 + 4 \cdot 0.66 + 0.74=8.42$$
and \eqref{eq:suff} holds for all $p\geq 2.4$. 
\end{proof}

When performing computations with Wasserstein distances, we used the
approximate Wasserstein distance algorithm implemented in
\textit{Hera}~\cite{Hera}. The algorithm computes the distances up to
arbitrarily chosen relative errors, that we set very low.

\subsection{Persistence landscapes} \label{sec:landscapes}

Recall that $H_*(a,b) := Z_*(K_a)/ (B_*(K_b) \cap Z_*(K_a))$ is the
vector space of non-trivial homology classes in $H_*(K_a)$ that are
still distinct when thought of as elements of $H_*(K_b)$ under the
induced map $H_*(K_a)\to H_*(K_b)$. For $a\leq b$ let
$\beta^{a,b}=\dim(H_*(a,b))$. We can think of
$\beta^{\bullet,\bullet}$ as a persistent version of the ordinary
Betti numbers. Indeed, $\beta^{a,a}$ is the Betti number of
$K_a$. Notably, persistent Betti numbers are non-negative integer
valued functions. Furthermore, when $a\leq b\leq c\leq d$, then
$\beta^{a,d}\leq \beta^{c,d}$. We can construct the persistence
landscape as a sequence of functions which together completely
describe the level sets of these functions.
\begin{definition}
  The \emph{persistence landscape} of some filtration is a function
  $\lambda:\mathbb{N}\times \R \to \overline{\R}$, where
  $\overline{\R}=[-\infty,\infty]$ denotes the extended real numbers,
  defined by
  \begin{equation*}
    \lambda_k(t)=\sup\{m\geq0\,|\,\beta^{t-m,t+m}\geq k\}.
  \end{equation*}
  We alternatively think of the landscapes as a sequence of functions
  $\lambda_k:\R\to \overline{\R}$ with $\lambda_k(t)=\lambda(k,t)$.
\end{definition}

Since persistence landscapes are real-valued functions, we can
consider the space of these functions with the $L^p$ norm
\begin{equation*}
  \|\lambda\|_p^p = \sum_{k=1}^\infty \|\lambda_k\|_p^p
\end{equation*}
for $1\leq p\leq \infty$.

\begin{theorem} \label{thm:Landscapes}
  The following are true for the space of persistence landscapes under
  different $L^p$ norms:
  \begin{enumerate}
  \item $p=2$: It is of strong negative type.
  \item $1\leq p \leq 2$: It is of negative type.
  \item $p=1$: It is not of strong negative type, even when restricting to
    persistence landscapes that arise from persistence diagrams.
  \item $p=\infty$: It is not of negative type, even when restricting
    to persistence landscapes that arise from persistence diagrams.
  \end{enumerate}
\end{theorem}

\begin{proof}\,
  \begin{enumerate}
  \item The space of persistence landscapes with the $L^2$ norm is a
    separable Hilbert space. Applying Theorem \ref{thm:SepHilbert}
    shows it is of strong negative type.
  \item As discussed in~\cite{Landscapes}, these function spaces are
    $L^p$ function spaces. From Theorem~3.6 in~\cite{Meckes} we know
    that these are of negative type when $1\leq p\leq 2$.
  \item The space of persistence landscapes with $L^1$ norm is of
    negative type but not of strong negative type. We can construct a
    counterexample using only distributions of landscapes that arise
    from persistent homology. To this end it is sufficient to provide
    appropriate barcodes, each with finitely many bars, as every such
    barcode can be realised. Let
    \begin{align*}
      X_1=I_{[0,1)} \oplus I_{[3,4)},& \qquad  Y_1=I_{[0,1)} \oplus I_{[1,2)},\\
      X_2=I_{[1,2)} \oplus I_{[2,3)},& \qquad  Y_2=I_{[2,3))}\oplus I_{[3,4)}.
    \end{align*}
    Since all the bars in each barcode are disjoint, only the first
    persistence landscape in non-zero. 
    
    Let $\PL(Z)$ denote the
    persistence landscape of $Z$, and $d_1$ the metric induced by
    the $L^1$ norm.
    
     We have $d_1(\PL(X_1),
    \PL(X_2))=2=d_1(\PL(Y_1),\PL(Y_2))$ and $d_1(\PL(X_i),\PL(Y_j))=1$ for all $i,j$. If
    we weight $X_1$ and $X_2$ by $1$, and the $Y_1$ and $Y_2$ by $-1$,
    then the weighted sum from inequality~\eqref{eq:negtype} is $0$,
    which means that the space of persistence landscapes with $L^1$
    norm is of non-strict negative type.
  \item For $p=\infty$ the space of persistence landscapes is not of
    negative type. We can construct a counterexample using only
    distributions of landscapes that arise from persistent
    homology. Again we do to this via barcodes. Let
    \begin{align*}
      X_1 & =I_{[0,2)} \oplus I_{[6.5,7.5)}\oplus I_{[8.5,9.5)}\oplus I_{[10.5,11.5)} \\
      Y_1 &= I_{[0.5,1.5)} \oplus I_{[2.5,3.5)}\oplus I_{[4.5,5.5)}\oplus I_{[6,8)}\\
      X_2 &= I_{[2,4)} \oplus I_{[6.5,7.5)}\oplus I_{[8.5,9.5)}\oplus I_{[10.5,11.5)}\\
      Y_2 &= I_{[0.5,1.5)} \oplus I_{[2.5,3.5)}\oplus I_{[4.5,5.5)}\oplus I_{[8,10)}\\
      X_3 &= I_{[4,6)} \oplus I_{[6.5,7.5)}\oplus I_{[8.5,9.5)}\oplus I_{[10.5,11.5)}\\
      Y_3 &= I_{[0.5,1.5)} \oplus I_{[2.5,3.5)}\oplus I_{[4.5,5.5)}\oplus I_{[10,12)}.
 \end{align*}
    Since all the bars in each barcode are disjoint, only the first
    persistence landscape is non-zero.
 
    It is straightforward to compute the $L^\infty$ distances between
    the corresponding persistence landscapes. Let $\PL(Z)$ denote the
    persistence landscape of $Z$, and $d_\infty$ the metric induced by
    the $L^\infty$ norm. We see that
    $d_\infty(\PL(X_i),\PL(X_j))=1=d_\infty(\PL(Y_i),\PL(Y_j))$ when
    $i\neq j$ and $d_\infty(\PL(X_i),\PL(Y_j))=0.5$ for all $i,j$. If
    we weight each of the $X_i$ with $1$ and the $Y_i$ by $-1$ we get
    the desired counterexample showing that the space of persistence
    landscapes with the $L^\infty$ distance is not of negative type.
 \end{enumerate}
\end{proof}
  
Persistence landscapes computations were performed using the
persistence landscapes toolkit~\cite{toolkit}.

\subsection{Persistence scale space kernel}
The persistence scale space kernel is a modification of scale space
theory to a persistence diagram setting. Extra care is needed to
consider the role of the diagonal. The idea is to consider the heat
kernel with an initial heat energy of Dirac masses at each of the
points in the persistence diagram with the boundary condition that it
is zero on the diagonal. The amount of time over which the heat
diffusion takes place is a parameter. More formally, it is defined
in~\cite{Heat} as follows.
\begin{definition}
  Let $\delta_p$ denote a Dirac delta centered at the point $p$. For a
  given finite persistence diagram $D$ with only finite
  lifetimes\footnote{When analyzing real data, one often cones off the
    space at some more or less meaningful maximum filtration so as to
    avoid infinite intervals.}, we now consider the solution
  $u:\R^{2+} \times \R_{\geq 0}\to \mathbb{R}$ of the partial
  differential equation
  \begin{align*}
    \Delta_x u&=\partial_t u &  \text{in } \R^{2+}\times \mathbb{R}_{>0},\\
    u&= 0 &  \text{on } \partial\R^{2+}\times \mathbb{R}_{\geq0} = \Delta\times\mathbb{R}_{\geq 0},\\
    u&= \sum_{p\in D} \delta_p& \text{on } \R^{2+}\times\{0\}.
  \end{align*}

\end{definition}

The solution $u(\bullet, t)$ lies in $L_2(\R^{2+})$ whenever $D$ has
finitely many points. It has a nice closed expression using the
observation that it is the restriction of the solution of a PDE with
an initial condition where below the diagonal we start with the
negative of the Dirac masses over the reflection of the points in the
diagram above the diagonal. For $x\in \R^{2+}$ and $t>0$ we
have
\begin{equation*}
  u(x,t)=\frac{1}{4\pi t}\sum_{(a,b)\in D}
  \left(\exp\left(\frac{-\|x-(a,b)\|^2}{4t}\right)
  -\exp\left(\frac{-\|x-(b,a)\|^2}{4t}\right)\right).
\end{equation*}
 
The metric for the space of persistence scale shape kernels is that of
$L^2(\mathbb{R}^{2+})$. The closed form for the persistence scale
space kernel allows a closed form of the pairwise distances in terms
of the points in the original diagrams. In particular for diagrams $F$
and $G$ and fixed $\sigma>0$, this distance can be written in terms of
a kernel $k_\sigma(F,G)$, where
\begin{equation*}
  k_\sigma(F,G)=\frac{1}{8\pi\sigma}\sum_{(a,b)\in F}\sum_{(c,d)\in G}\left( \exp\left(\frac{-\|(a,b)-(c,d)\|^2}{8\sigma}\right)-\exp\left(\frac{\|(a,b)-(d,c)\|^2}{8\sigma}\right)\right)
\end{equation*}
and the corresponding distance function is
\begin{equation*}
  d(F, G)= \sqrt{k_\sigma(F,F)+k_\sigma(G,G)-2k_\sigma(F,G)}.
\end{equation*}

Since $L^2(\mathbb{R}^{2+})$ is a separable Hilbert space, this metric
is of strong negative type.

%in paper:
%Unfortunately, as observed experimentally in Appendix A of [33], d1 is not conditionally negative definite (in practice, it only suffices to sample a family of point clouds to observe experimentally that more often than not the inequality above will be violated for a particular weight vector a). Actually, as observed in [30], even the square of the diagram distances dp cannot be used to define Gaussian kernels. Indeed, it was noted in Theorem 6 of [16] that, if the square of a distance d defined on a geodesic space X is conditionally negative definite, then the metric space X is flat,
 
%or CAT(0). However, since the metric space D, equipped with dp, p ? N ? {+?}, is not CAT(k) for any k > 0?which is due to the non-uniqueness of geodesics, see [40]?it follows that d2p is not conditionally negative definite.

\subsection{Sliced Wasserstein kernel distance}

The sliced Wasserstein distance between persistence diagrams,
introduced in~\cite{Sliced}, works with projections onto lines through
the origin. For each choice of line, one intuitively computes the
Wasserstein distance between the two projections (a computationally
much easier problem, being a matching of points in one dimension), and
then integrates the result over all choices of lines. More formally
the definition in~\cite{Sliced} is as follows.
\begin{definition}
  Given $\theta\in \mathbb{R}^2$ with $\|\theta\|_2=1$, Let
  $L(\theta)$ denote the line $\{\lambda \theta: \lambda\in
  \mathbb{R}\}$, and let $\pi_\theta:\mathbb{R}^2\to L(\theta)$ be the
  orthogonal projection onto $L(\theta)$. Let $D_1$ and $D_2$ be two
  persistence diagrams, and let $\mu_i^\theta=\sum_{p\in
    D_i}\delta_{\pi_\theta(p)}$ and $\mu_{i\Delta}^\theta=\sum_{p\in
    D_i} \delta_{\pi_\theta\circ
    \pi_{(\frac{1}{\sqrt{2}},\frac{1}{\sqrt{2}})}(p)}$ for
  $i=1,2$. Then the \emph{sliced Wasserstein distance} is defined as
  \begin{equation*}
    \SW(D_1, D_2) = \frac{1}{2\pi} \int_{S^1} \mathcal{W}(\mu_1^\theta+\mu_{2\Delta}^\theta, \mu_2^\theta+\mu_{1\Delta}^\theta)d\theta
  \end{equation*}
  where the $1$-Wasserstein distance $\mathcal{W}(\mu, \nu)$ is
  defined as $\inf_{P\in \Pi(\mu,\nu)}
  \int\int_{\mathbb{R}\times\mathbb{R}}|x-y|P(dx,dy)$ where $\Pi(\mu,
  \nu)$ is the set of measures on $\mathbb{R}^2$ with marginals $\mu$
  and $\nu$.
\end{definition}

It was shown in~\cite{Sliced} that the sliced Wasserstein distance is
conditionally seminegative definite on the space of finite and bounded
persistence diagrams. This is equivalent to the condition of being of
negative type. It is an open question as to whether it is of strong
negative type.

In~\cite{Sliced}, the authors construct a kernel with bandwidth
parameter $\sigma>0$ in the standard way (see~\cite{harmonic}), namely
\begin{equation*}
  k_\sigma(D_1, D_2) = \exp\left(\frac{-\SW(D_1,D_2)}{2\sigma^2}\right).
\end{equation*}
It being a kernel in the sense that
\begin{equation*}
  k_\sigma(D_1, D_2) = \left\langle \phi(D_1), \phi(D_2) \right\rangle_{\mathcal{H}}
\end{equation*}
for some function $\phi$ into a Hilbert space $\mathcal{H}$, one
obtains a distance function $d_{\text{kSW}}$ with
\begin{equation*}
  d_{\text{kSW}}(D_1, D_2)^2 =k_\sigma(D_1, D_1) + k_\sigma(D_2) - 2k_\sigma(D_1, D_2).
\end{equation*}
If this reproducing kernel Hilbert space $\mathcal{H}$ is separable,
then the space of persistence diagrams with $d_{\text{kSW}}$ will be
of strong negative type. This separability property is an open
question.

In our computations, we always projected onto $10$ equidistributed lines.

\section{Distance correlation between different topological summaries}

The differences between the metrics used can dramatically affect the
statistical analysis of a data set. It is important to choose a
summary such that the domain-specific differences in the input data
that are of interest are reflected in the distances between their
corresponding topological summaries.

The key idea in this section is to take the same object, for example
generated through a random process, and then to record different
topological summaries of it. As we have seen, this gives us different
metric space structures on the data.  We then compare the pairwise
distances using distance correlation.

We consider a variety of more or less standard or well known families
of random cell complexes and their filtrations, as well as some
non-random data.

\subsection{Erdős--Rényi}

We constructed the weighted version of $100$-vertex Erdős--Rényi
random graphs, which is to say we endow the complete graph on $100$
vertices with uniform random independent edge weights. The flag
complexes of each of these are then the filtrations we consider. We
generated $100$ such filtrations to get sample the distribution of
degree-$1$ persistent homology of such complexes. An example
persistence diagram is shown in Figure~\ref{fig:ER}. We then computed
the distance correlation between the different topological summaries,
with the result shown in Figure~\ref{fig:ER}.
\begin{figure}[htbp]
  \centering
  \includegraphics{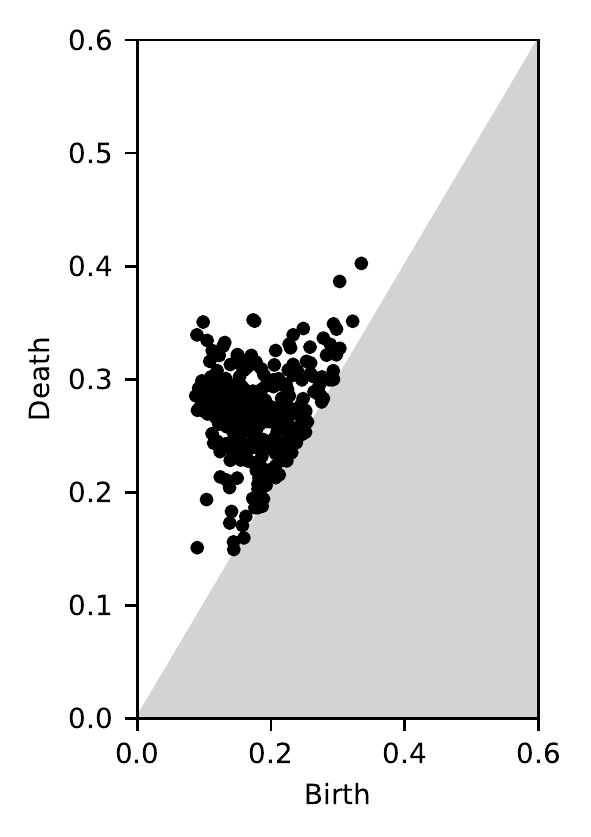}
  \includegraphics{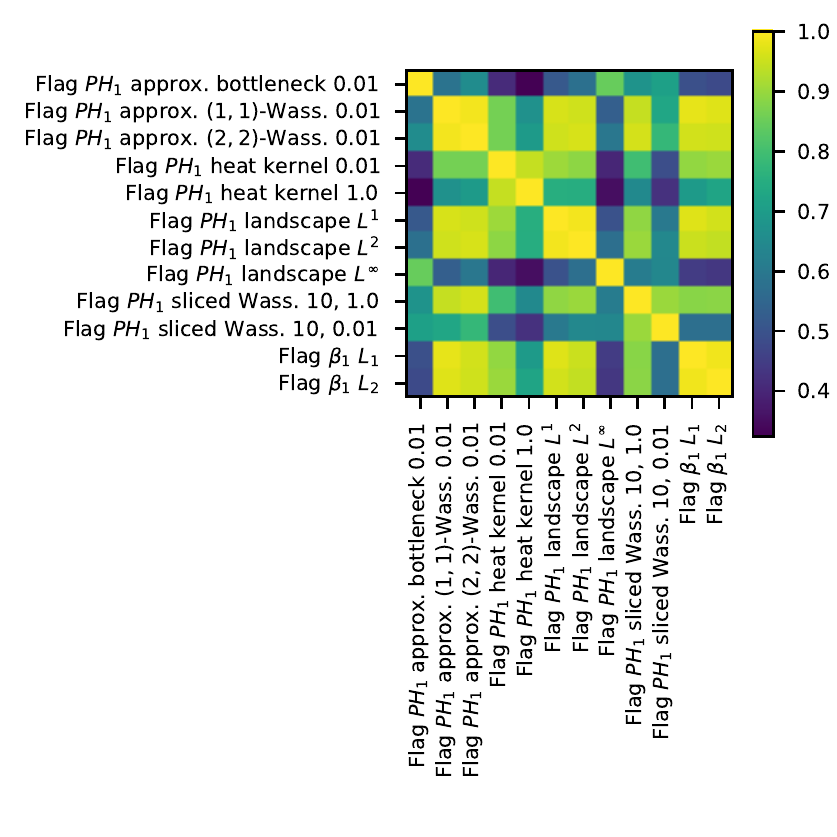}
  \caption{A typical degree-$1$ persistence diagram (\textbf{left})
    and the sampled square root of distance correlation ($\dCor$)
    between different topological summaries for Erdős--Rényi complexes
    (\textbf{right}).}
  \label{fig:ER}
\end{figure}

The persistent homology computations were performed using
\textit{Ripser}~\cite{ripser}.

\subsection{Directed Erdős--Rényi}

A directed analog of the flag complex of undirected graphs was
introduced in~\cite{BBP}. To construct such flag complexes, we
generated $100$ instances of the independently random uniform weights
on the complete directed graph on $100$ vertices (taking ``complete
directed graph'' to mean having opposing edges between every pair of
vertices), and computed the corresponding filtrations and
degree-$1$ persistent homology of directed flag complexes using
\textit{Flagser}~\cite{flagser}. An example persistence diagram is
shown in Figure \ref{fig:DER}. We then computed the distance
correlation between the different topological summaries, with the
result shown in Figure~\ref{fig:DER}.
\begin{figure}[htbp]
  \centering
  \includegraphics{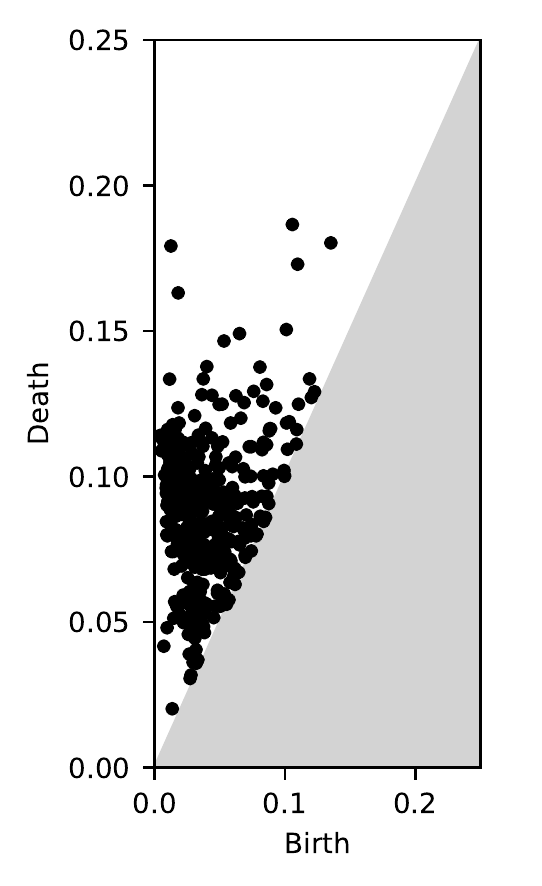}
  \includegraphics{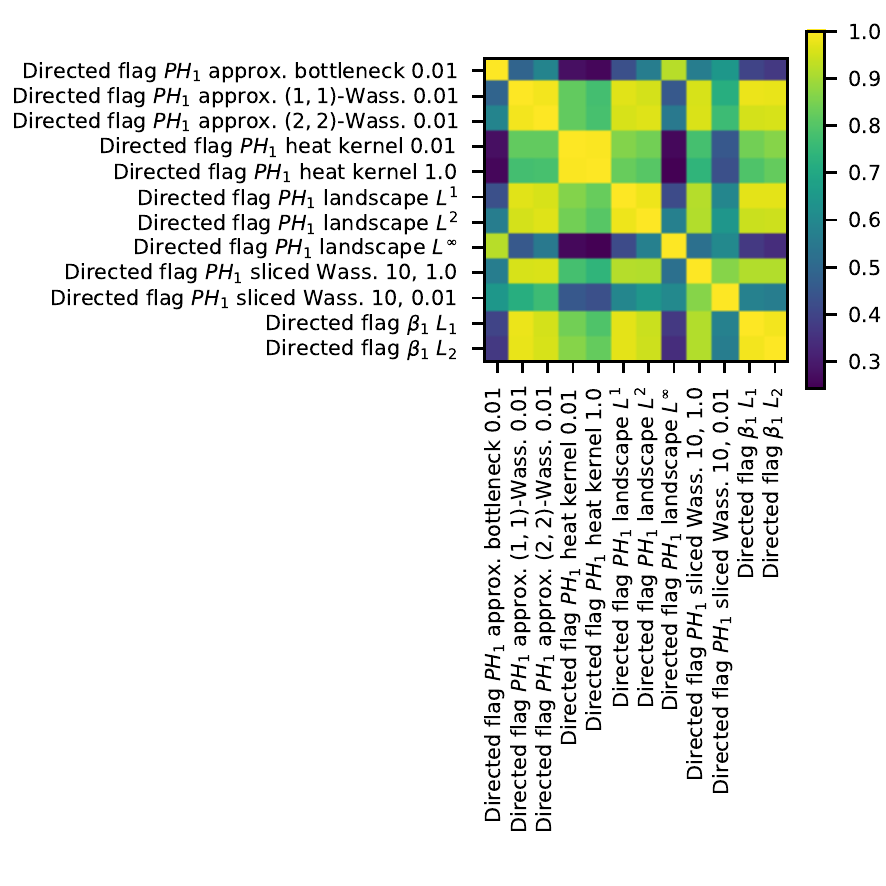}
  \caption{A typical degree-$1$ persistence diagram (\textbf{left})
    and the sampled square root of distance correlation ($\dCor$)
    between different topological summaries for directed Erdős--Rényi
    complexes (\textbf{right}).}
  \label{fig:DER}
\end{figure}

\subsection{Geometric random complexes for points sampled on a torus}

For this dataset, we randomly sampled $500$ points
independently from a flat torus in $\R^4$ by sampling $[0,2\pi)^2$
uniformly and considering the image of $(s,t)\mapsto(\cos s, \sin s,
\cos t, \sin t)$. We then built the alpha complex over this set
of points. This was performed $100$ times to construct samples of
the distribution of persistent homology in degree $1$ for such complexes. An
example persistence diagram is shown in
Figure~\ref{fig:GeomTorus}. We then computed the distance
correlation between the different topological summaries, which is
shown in Figure~\ref{fig:GeomTorus}.
 \begin{figure}[htbp]
  \centering
  \includegraphics{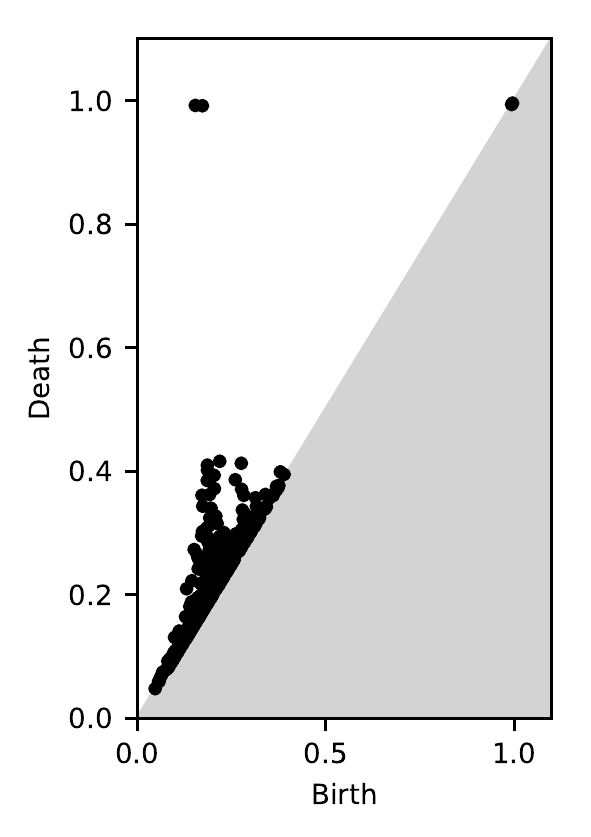}
  \includegraphics{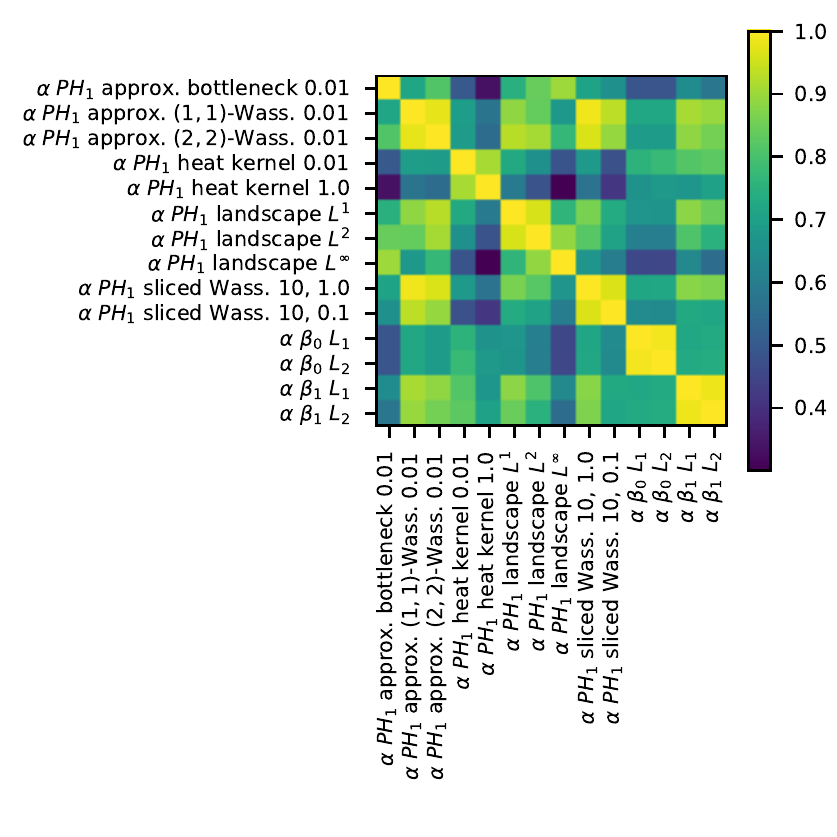}
  \caption{A typical degree-$1$ persistence diagram (\textbf{left})
    and the sampled square root of distance correlation ($\dCor$)
    between different topological summaries (\textbf{right}) for alpha
    complexes built from random points clouds sampled from a flat
    torus lying in $\R^4$.}
  \label{fig:GeomTorus}
\end{figure}

The computations of alpha complexes and persistent homology for this
dataset were done using \textit{GUDHI}~\cite{GUDHI}.

\subsection{Geometric random complexes for point sampled from a unit cube}

For this dataset, we uniformly randomly sampled $500$ points
independently from the unit cube $[0,1]^3$. We then constructed the
alpha complex over this set of points. This was performed $100$ times
to sample the distribution of persistent homology for such
complexes. An example persistence diagram is in
Figure~\ref{fig:GeomCube}. We then computed the distance correlation
between the different topological summaries which is shown in
Figure~\ref{fig:GeomCube}.
\begin{figure}[htbp]
  \centering
  \includegraphics{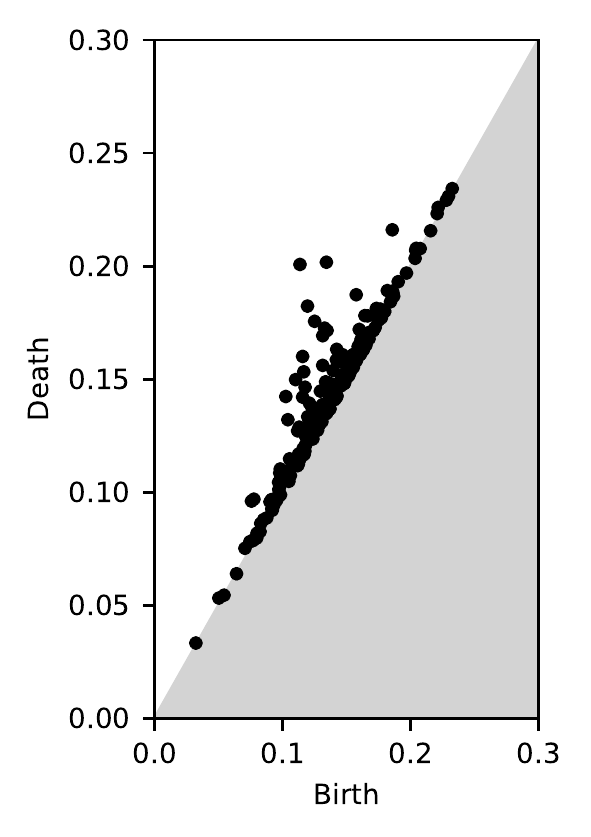}
  \includegraphics{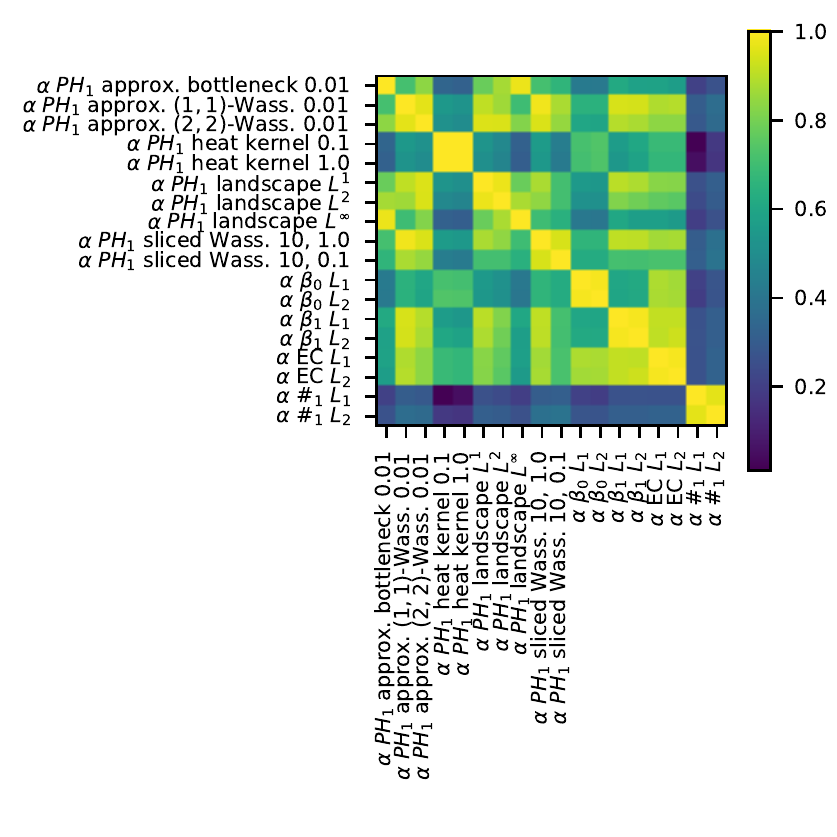}
  \caption{A typical degree-$1$ persistence diagram and the sampled
    square root of distance correlation ($\dCor$) between different
    topological summaries for alpha complexes built from random points
    clouds sampled from the unit cube.}
  \label{fig:GeomCube}
\end{figure}

For this particular dataset, we also computed a very non-topological
summary based on the same underlying complex, namely the counts of
$1$-simplices ($\#_1$ in Figure~\ref{fig:GeomCube}). These were
considered as ``count curves'' in the obvious way, and endowed with
the $L^1$ and $L^2$ metrics. They, unsurprisingly, correlate little
with the topological summaries.

The computations of the alpha complexes and persistent homology for
this dataset were done using \textit{GUDHI}~\cite{GUDHI}.

\subsection{Observations about the distance correlation in these simulations}

The first general comment is that the sampled distance correlations for the tological summaries split these different simulations into two groups; one group contains the directed Erd\"os R\'enyi filtrations and the Erd\"os R\'enyi filtrations, and the other group of simulations contain filtrations built from random point clouds either on the flat torus on is the unit cube. This is not too surprising as the both directed Erd\"os R\'enyi filtrations and the Erd\"os R\'enyiscenarios represent completely random types of complexes without correlations on the simplex values. In contrast, for fitrations built on point clouds there are geometric constriants which imply correlations between the filtration values on neighbouring simplices. This in turn affects the observed topology.

For both the persistence diagrams and the persistence landscapes sampled the distance correlations for $p=1,2$ and $\infty$. In all of the simulation studies the metrics from $p=1$ and $p=2$ of the same topological summary generally have high distance correlation, but that they are quite different to the $p=\infty$ version of that same topological summary. This is particularly prenounced in the Erd\"os R\'enyi and directed Erd\"os R\'enyi filtrations. In fact here the distance correlation between digrams and landscapes with $p=1$ and $p=2$ is higher than the distance correlation between bottleneck distances and $p$-Wasserstein distances for $p=1$ or $p=2$, and similarly for landscapes. One explanation is that in the completely random scenario we can have more extremal persistent homology classes and these extremal persistent homology classes dominate the $p=\infty$ metrics more that in the $p=1$ and $p=2$ metrics. 

Another observation is that overall we see high correlation between the Sliced Wasserstein distances and the Wasserstein ($p=1$ or $p=2$) distances. Perhaps not surprising since both are geometrically measuring similar quantities with a pairing process of points involved in both distances (though the pairing potentially varying between slices in the Sliced Wasserstein).

\section{Distance correlation to another parameter}

Instead of considering the correlation between two distances of
topological summaries, one may want to consider the correlation
between a metric on topological summaries and some real number
relating to the underlying model. The real number may for example
parameterise the underlying model, or it may be some function of the
model that has domain-specific meaning. We will here consider only
parameters and functions with codomain in (intervals in) $\R$, and
consider that space as a metric space equipped with the absolute value
distance.

For brevity, we will from now on refer also to the value of certain
domain-specific functions on the underlying model as ``parameters'',
even though they strictly speaking are not (see for example the case
of elevation data below, where terrain smoothness will incorrectly be
referred to as a parameter of the landscape). We will also use the
letter $\mathcal{P}$ to denote the parameter space as a metric space
with the absolute value distance.

We can use distance correlation to quantify how well the distances
between some topological summaries relate to the differences in the
parameter. The varying performances of the different topological
summaries in correlating to the parameter highlights how the choice of
topological summary has statistical significance.

%% In this study we consider a model of random filtrations with a
%% parameter $\gamma \in [0,1]$ such that for $\gamma=0$ we have a
%% geometric Rips filtration, and for $\gamma=1$ we have an Erdos-Renyi
%% Rips filtration. We compute the distance correlation of the
%% topological summaries to this parameter $\gamma$. We also consider
%% some topographic data from maps of Norway and compute the distance
%% correlation of topological summaries of the elevation function to
%% geographic measurements of geodesic measurements between the locations
%% of the maps and also distance correlation to the Terrain Roughness
%% Index.

\subsection{Parameterised interpolation between Erdős--Rényi and geometric complexes}

Our parameter space is now $[0,1]$. Each sampled filtered complex with
parameter $\gamma\in[0,1]$ is built by sampling $100$ points
$X=\{x_1,\dotsc, x_{100}\}$ i.i.d.\ uniformly from the unit cube
$[0,1]^3$, and sampling the entries of a symmetric matrix
$E\in\R^{100\times 100}$ i.i.d.\ uniformly from $[0,1]$. We endow a
complete graph on $100$ vertices with weights $w_{i,j}$ for each pair
$1\leq i<j\leq 100$ by letting $w_{i,j}=E_{i,j}$ with probability
$\gamma$ and $w_{i,j}=\|x_i-x_j\|$ with probability $1-\gamma$. The
filtered complex generated is then the flag complex of this
graph. Observe that this is a (Vietoris--Rips) version of the random
geometric complex considered before when $\gamma=0$, and the
Erdős--Rényi complex when $\gamma=1$.

A correlation between the parameter space and a given metric on a
topological summary is then a measure of how well that metric detects
the parameter.

For this experiment, we let $\gamma$ take the $100$ equally spaced
values from $[0,1]$, including the endpoints. These distance
correlations are displayed in Table~\ref{tab:ER-Geom}. The higher the
distance correlation, the better the topological summary reflects the
effect of the parameter $\gamma$. We see that generally the function
distances between Betti curves, Wasserstein distances and bottleneck
distances between and the function distances between persistence
landscapes had a higher correlation, all with a distance correlation
greater than $0.9$. This illustrates that these topological summaries
would be good choices if we wish to do learning problems or
statistical analysis with regards to this parametrised random model,
such as parameter estimation. We also see the importance of the choice
of bandwidth with dramatic effect on the distance correlation of the
persistence scale space kernel and the Sliced Wasserstein kernel.

\begin{table}[htbp]
  \begin{center}
    \begin{tabular}{l|c}
      Topological summary                                    & $\dCov(\bullet, \mathcal{P})$ \\
      \hline
      Persistence scale space kernel, $\sigma=0.001$         & $0.96$                      \\
      $1$-Wasserstein                                        & $0.95$                      \\
      $\beta_1$ with $L^1$                                   & $0.95$                      \\
      $\beta_1$ with $L^2$                                   & $0.95$                      \\
      $2$-Wasserstein                                        & $0.94$                     \\
      Persistence landscape with $L^\infty$                   & $0.94$                      \\
      Persistence scale space kernel, $\sigma=0.01$          & $0.93$                     \\
      Persistence landscape with $L^2$                      & $0.92$                      \\
      Persistence landscape with $L^1$                       & $0.92$                     \\
      Sliced Wasserstein kernel, $\sigma=1$                 & $0.66$                      \\
      Persistence scale space kernel, $\sigma=1$            & $0.60$                      \\
      Sliced Wasserstein kernel, $\sigma=0.01$              & $0.40$
    \end{tabular}
    \caption{(Square roots of) distance correlation between topological summaries and the parameter $\gamma$.}
    \label{tab:ER-Geom}
  \end{center}
\end{table}

The persistent homology computations were performed using
\textit{Ripser}~\cite{ripser}.

\subsection{Digital elevation models and terrain ruggedness} \label{sec:dem}

As a simple example of ``real world'' data, we considered digital
elevation model (DEM) data for a $50$ km by $50$ km patch around the
city of Trondheim, Norway\footnote{The data was provided by the
  Norwegian Mapping Authority~\cite{Maps} under a CC-BY-4.0
  license.}. The DEM data set maps elevation data with a horizontal
resolution of $10$ m $\times$ $10$ m and a vertical resolution of
about $1$ m, and as such provides a terrain height
map. Figure~\ref{fig:Map} shows the DEM our data was based on.
\begin{figure}[htbp]
  \centering
  \includegraphics[width=100mm]{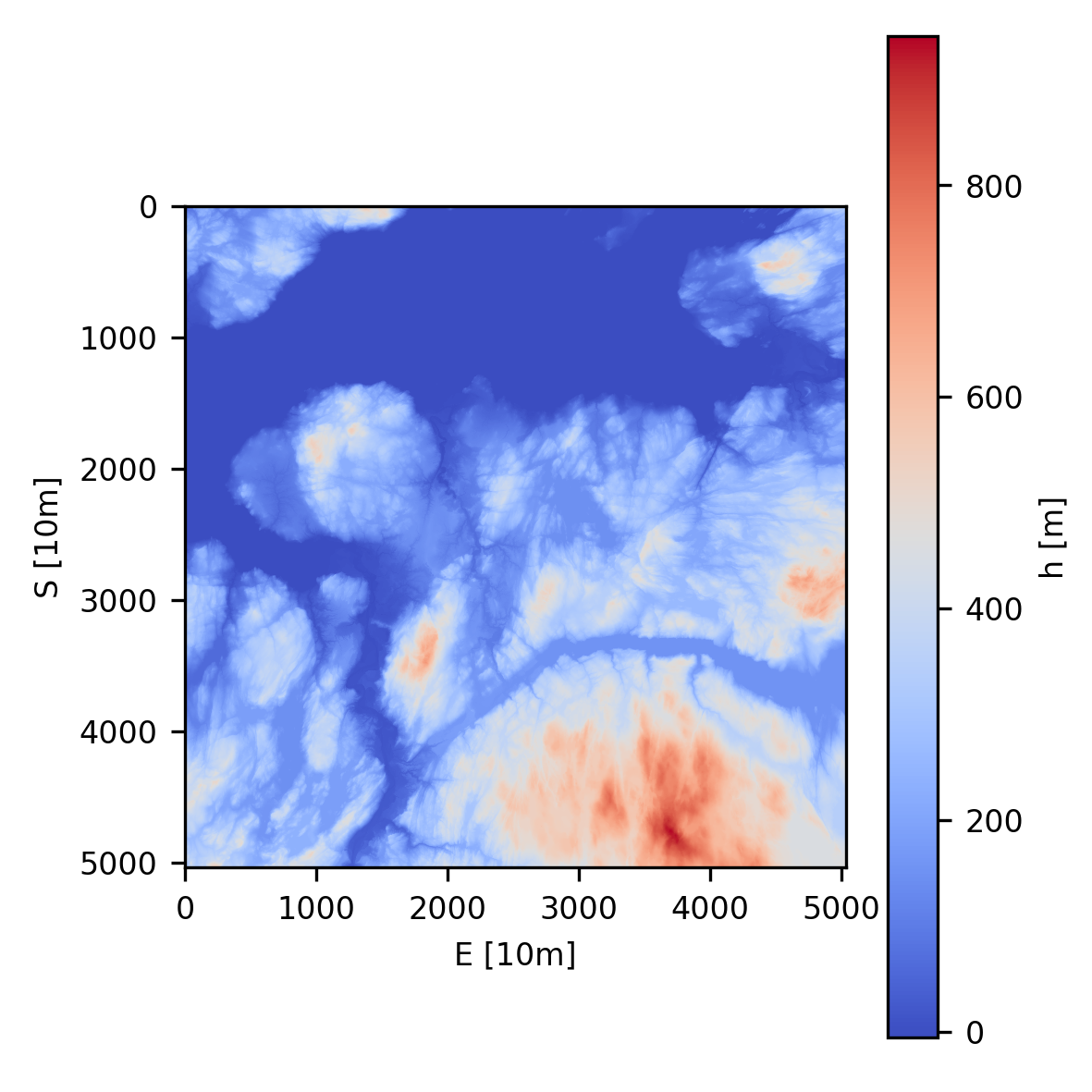}
  \caption{The digital elevation model near Trondheim, Norway. Each
    filtration we construct is based on one of $64$ $1000\times1000$
    sub-blocks of this DEM.}
  \label{fig:Map}
\end{figure}

The data can, at the aforementioned horizontal resolution, be
considered as a $5000\times5000$ integer-valued matrix $Z$, where each
entry is interpreted as the height above a reference elevation in the
vertical resolution unit. Each filtered complex we consider comes from
$1000\times 1000$ block in $Z$. The blocks overlap up to $50\%$, and
we keep a total of $64$ blocks. Each block is then considered as a
two-dimensional cubical complex with the elevation data as the height
filtration on the $2$-cells.

The \emph{terrain ruggedness indicator (TRI)} is an extremely simple
measure of local terrain ruggedness that is widely employed in GIS and
topography~\cite{TRI}. The TRI itself is a real-valued function
defined on each map point/pixel, and a high value indicates a locally
more rugged terrain. We simplify the measure even further by averaging
the TRI for the whole map chunk considered, thus assigning a single
real number to each map chunk. It is this number that will play the
role of one parameter assigned to each of the $64$ map chunks
considered, which we will call $\mathrm{TRI}$. The results are shown
in Table~\ref{table:DEM}.

Another natural metric that can be defined on the raw data (the
$1000\times 1000$ chunks) itself is the actual geodesic distance
between the centers of the chunks. We also computed distance
correlations between the metrics on topological summaries and this
geodesic distance, although one must remember that it is perhaps not
reasonable to expect a high correlation here; indeed, topographies
with topologically highly interesting height functions may exist on a
coastline, and thus be very close topologically trivial terrain. The
results in Table~\ref{table:DEM} are therefore quite surprising.

\begin{table}[htbp]
  \begin{center}
    \begin{tabular}{l|c|c}
      Summary (topological, apart from first two)     & $\dCor(\bullet, \mathrm{TRI})$ & $\dCor(\bullet, d_{\text{geodesic}})$ \\
      \hline
      $\mathrm{TRI}$                                & $1$                            & $0.72$                            \\
      $d_{\text{geodesic}}$                             & $0.72$                         & $1$                               \\
      Bottleneck                                    & $0.62$                          & $0.52$                           \\
      $2$-Wasserstein                               & $0.92$                         & $0.74$                            \\
      Persistence scale space kernel, $\sigma=1$    & $0.74$                         & $0.64$                            \\
      Persistence scale space kernel, $\sigma=10$   & $0.75$                         & $0.63$                            \\
      Persistence landscape with $L^1$              & $0.73$                         & $0.61$                                  \\
      Persistence landscape with $L^2$              & $0.72$                         & $0.59$                                  \\
      Persistence landscape with $L^\infty$          & $0.62$                         & $0.52$                           \\
      Sliced Wasserstein kernel, $\sigma=1$        &  $0.44$                         & $0.55$  \\
      Sliced Wasserstein kernel, $\sigma=0.01$     & $0.44$                          & $0.55$  \\
      $\beta_1$ with $L^1$                           & $0.75$                        & $0.65$                            \\
      $\beta_1$ with $L^2$                           & $0.77$                         & $0.63$
    \end{tabular}
    \caption{(Square roots of) distance correlation between elevation
      topological summaries and the terrain ruggedness index and the
      geodesic distance.}\label{table:DEM}
  \end{center}
\end{table}

The cubical complex persistent homology calculations were doing using
\textit{GUDHI}~\cite{GUDHI}.

\section{Future directions} \label{sec:future}

Non-parametric statistics is a fruitful area for ideas and inspiration
for methods that can be applied in conjunction with TDA. There are
already a variety of options that only use pairwise distances,
including null-hypothesis testing, clustering, classification, and
parameter estimation. In all these cases, we would expect that
distance correlation would be a good estimator for similarity of
statistical analyses.

We can perform null hypothesis testing with topological summaries via
a permutation text with a loss function a function of the pairwise
distances (see \cite{Permutation}). Intuitively, when there is a high
distance correlation, the pairwise distances are correlated and the
corresponding loss functions should be similar for each permutation of
the labels. This implies we should expect that the $p$-values given a
sample distribution should be close, at least with high
probability. It may be possible to show that the power of the null
hypothesis tests are close. An experimental and theoretical
exploration of this relationship is a future direction.

We can also think of considering a modification of the permutation
test for independence using distance correlation (instead of Pearson
correlation). This can then be applied to topological summaries. One
can get a $p$-value that for whether two variables are independent by
permuting the coupling of the variables but keeping the marginal
distributions the same. A high ranking of the distance correlation for
the original joint distribution would indicate that the variables are
not independent, with high probability. Exploring the power of this is
a future direction of research.

Another non-parametric method is parameter estimation using nearest
neighbours. One method for estimating a real valued parameter which is
unknown on a particular sample, but is known on a training set, is to
take a weighted average of the values of the parameter on the training
set with the weighting dependent on the pairwise distances from the
sample of interest to those in the training set. We would expect
better estimation when the distance correlation between the samples
and the parameter of interest is high. Future directions for research
can include experimental and theoretical results along these lines
with respect to topological summaries. In particular, we would expect
that we should be able to create confidence intervals for the
parameter, dependent on the distance correlation. This is also an area
where we should expect similar statistical analysis when the samples
have high distance correlation.

Completely analogous to the above comments, clustering methods using
pairwise distances should have similar results when the sets of
samples have high distance correlation and future work could explore
this with respect to topological summaries.

\section*{Acknowledgments}

G.S.\ would like to thank Andreas Prebensen Korsnes of the Norwegian
Mapping Authority for going out of his way to facilitate bulk
downloads of DEM data before a single region was decided upon for the
experiment in section~\ref{sec:dem}.

G.S.\ was supported by Swiss National Science Foundation grant number 200021\_172636.

\printbibliography

\end{document}